\documentclass[11pt]{article}
\usepackage{amsmath, amsfonts, amsthm, amssymb, color}
\usepackage{graphicx}
\usepackage{float}
\usepackage{verbatim}

\usepackage{subcaption}

\usepackage{bbm}

\allowdisplaybreaks

\numberwithin{equation}{section}

\newtheorem{lemma}{Lemma}[section]

\newtheorem{prop}[lemma]{Proposition}
\newtheorem{thm}[lemma]{Theorem}
\newtheorem{cor}[lemma]{Corollary}
\theoremstyle{definition}

\theoremstyle{remark}
\newtheorem{remark}[lemma]{Remark}

\def\R{\mathbb{R}}
\def\N{\mathbb{N}}

\def\a{\mathbf{a}}
\def\b{\mathbf{b}}
\def\c{\mathbf{c}}
\def\m{\mathfrak{m}}

\def\p{\mathbf{p}}

\def\A{\mathcal{A}}
\def\pp{\mathfrak{p}}

\numberwithin{equation}{section} \numberwithin{table}{section}

\title{Intrinsic Diophantine Approximation for overlapping iterated function systems}
\author{Simon Baker\\ \\
\emph{School of Mathematics,} \\ \emph{University of Birmingham,} \\ \emph{Birmingham,  B15 2TT, UK.} \\ Email: simonbaker412@gmail.com\\}

\date{\today}
\begin{document}
\maketitle

\begin{abstract}
In this paper we study a family of limsup sets that are defined using iterated function systems. Our main result is an analogue of Khintchine's theorem for these sets. We then apply this result to the topic of intrinsic Diophantine Approximation on self-similar sets. In particular, we define a new height function for an element of $\mathbb{Q}^d$ contained in a self-similar set in terms of its eventually periodic representations. For limsup sets defined with respect to this height function, we obtain a detailed description of their metric properties. The results of this paper hold in arbitrary dimensions and without any separation conditions on the underlying iterated function system. \\



\end{abstract}

\section{Introduction}

Diophantine Approximation is the study of approximations of vectors in $\mathbb{R}^d$ by elements of $\mathbb{Q}^d$. Given a set $X\subset \mathbb{R}^d$, it is natural to wonder how well elements of $X$ can be approximated by elements of $\mathbb{Q}^d$ contained within $X$. Similarly, it is natural to wonder how well elements of $X$ can be approximated by elements of $\mathbb{Q}^d$ lying outside of $X$. These two questions are the motivation behind the topics of intrinsic Diophantine Approximation and extrinsic Diophantine Approximation respectively. Often the set $X$ is taken to be a smooth manifold or a fractal set. A tremendous amount of work has been done on these two topics when $X$ is taken to be such a set. For further details we refer the reader to the papers \cite{Ber,BVVZ,BV,BFR,Bug2,BugDur,FKMS,FMS,FS,FS2,KL,KLW,Sch,SW,TWW,Weiss,Yu} and the references therein. In this paper we study intrinsic Diophantine Approximation when the set $X$ is a self-similar set. We will provide a more thorough introduction to this topic in Section \ref{Application section}. The main result of this paper is a general theorem on the metric properties of a family of limsup sets defined using iterated function systems. As we will see, this theorem implies a number of results in intrinsic Diophantine Approximation.

In what remains of this introductory section we will provide the relevant background from Fractal Geometry and state Theorem \ref{Main theorem}, which is our main result. In Section \ref{Application section} we will show how Theorem \ref{Main theorem} can be used to obtain a number of results for intrinsic Diophantine Approximation on self-similar sets. In Section \ref{Proof section} we will prove Theorem \ref{Main theorem}. In Section \ref{MTP section} we will apply the mass transference principle of Beresnevich and Velani together with Theorem \ref{Main theorem} to deduce further results on the Hausdorff measure of certain limsup sets.

\subsection{Background from Fractal Geometry}
We call a map $\phi:\mathbb{R}^{d}\to\mathbb{R}^d$ a similarity if there exists $r\in(0,1)$ such that $\|\phi(x)-\phi(y)\|=r\|x-y\|$ for all $x,y\in\mathbb{R}^d$. We call a finite set of similarities an iterated function system or IFS for short. An important result due to Hutchinson \cite{Hut} states that for any IFS $\Phi=\{\phi_{a}\}_{a\in\mathcal{A}},$ there exists a unique non-empty compact set $X$ satisfying $$X=\bigcup_{a\in \mathcal{A}}\phi_{a}(X).$$ $X$ is called the self-similar let of $\Phi$. When the elements of $\Phi$ all have the same contraction ratio, i.e. $r_{a}=r_{a'}$ for all $a,a'\in \A$, then we say that an IFS is equicontractive. Importantly we can view $X$ as the image of $\A^{\N}$ under an appropriate projection map: Let $\pi:\A^{\N}\to X$ be given by $$\pi((a_n)_{n=1}^{\infty})=\lim_{n\to\infty}(\phi_{a_1}\circ \cdots \circ \phi_{a_n})(0).$$ Here $0$ can be replaced with any other vector in $\mathbb{R}^d$. Importantly the map $\pi$ is surjective and continuous (when $\A^{\N}$ is equipped with the product topology). Given an IFS $\Phi=\{\phi_{a}\}_{a\in\mathcal{A}}$ we define the similarity dimension of $\Phi$ to be the unique solution to the equation $$\sum_{a\in \mathcal{A}}r_{a}^{s}=1.$$ We denote the similarity dimension of an IFS $\Phi$ by $\dim_{S}(\Phi)$. Notice that if $\Phi$ is equicontractive then $\dim_{S}(\Phi)=\frac{\log \#\A}{-\log r}$ where $r$ is the common contraction ratio. It is well known that the Hausdorff dimension of a self-similar set $X$ always satisfies the following upper bound:
\begin{equation}
\label{similarity inequality}
\dim_{H}(X)\leq \min\{\dim_{S}(\Phi),d\}.
\end{equation}For many iterated function systems this inequality is in fact an equality, see \cite{Fal,Hochman,Hochman2,Rap}. We say that $\Phi$ satisfies the strong separation condition if $\phi_{a}(X)\cap \phi_{a'}(X)=\emptyset$ for all $a,a'\in \A$ such that $a\neq a'$. An IFS $\Phi$ is said to satisfy the open set condition if there exists a bounded open set $O$ such that $\phi_{a}(O)\subset O$ for all $a\in \A$, and $\phi_{a}(O)\cap \phi_{a'}(O)=\emptyset$ whenever $a\neq a'$. It is known that the strong separation condition implies the open set condition, and that under either of these assumptions we have equality in \eqref{similarity inequality}.

To prove equality in \eqref{similarity inequality} in the overlapping case one often uses self-similar measures. These are defined as follows: Given an IFS $\{\phi_{a}\}_{a\in A}$ and a probability vector $\p=(p_{a})_{a\in \A},$ then there exists a unique Borel probability measure $\mu_{\p}$ satisfying $$\mu_{\p}=\sum_{a\in \A}p_{a}\cdot \phi_{a}\mu_{\p}.$$ We call $\mu_{\p}$ the self-similar measure corresponding to $\Phi$ and $\p$. Given $\p$, if we let $\m_{\p}$ denote the corresponding Bernoulli measure on $\A^{\N}$ then it is also the case that $\mu_{\p}=\pi \m_{\p}$. For our purposes we will only need to focus on one particular self-similar measure, namely the one corresponding to the probability vector $(r_{a}^{\dim_{S}(\Phi)})_{a\in \A}.$ This self-similar measure is distinguished amongst the family of self-similar measures. Studying its properties often allows one to prove equality in \eqref{similarity inequality}. For an IFS $\Phi$, we will denote the self-similar measure corresponding to $(r_{a}^{\dim_{S}(\Phi)})_{a\in \A}$ by $\mu_{\Phi},$ or simply $\mu$ if the choice of $\Phi$ is implicit. Similarly, we will denote $(r_{a}^{\dim_{S}(\Phi)})_{a\in \A}$ by $\p_{\Phi}$ or simply $\p$, and the corresponding Bernoulli measure on $\A^{\N}$ by $\m_{\Phi}$ or $\m$. For a probability vector $\p$ we denote the entropy of $\p$ by $$h_{\p}:=-\sum_{a\in \A} p_{a}\log p_{a}.$$ Suppose now that in addition to $\p$ we are also given an IFS $\Phi$, we then define the Lyapnuov exponent of $\Phi$ and $\p$ to be $$\chi_{\Phi,\p}:=-\sum_{a\in \A}p_{a}\log r_{a}.$$ 

We conclude this overview of the relevant topics from Fractal Geometry by introducing some notation. In what follows, we denote an element of $\cup_{n=1}^{\infty}A^n$ or $\A^{\N}$ by $\a$ or $\b$. Given an IFS $\Phi=\{\phi_{a}\}_{a\in \A}$ and a word $\a=(a_1,\ldots,a_n),$ we let $\phi_{\a}:=\phi_{a_1}\circ \cdots \circ \phi_{a_n}$ and $r_{\a}:=\prod_{l=1}^{n}r_{a_l}.$ Given a word $\a$ we let $X_{\a}=\phi_{\a}(X)$. Given a finite word $\a$ and a finite word or infinite sequence $\b,$ we let $\a\b$ denote the concatenation of $\a$ and $\b$. For a finite word $\a$ we let $\a^k$ denote the $k$-fold concatenation of $\a$ with itself. Similarly $\a^{\infty}$ denotes the periodic element of $\A^{\mathbb{N}}$ obtained by concatenating $\a$ with itself indefinitely. We denote the length of a finite word $\a$ by $|\a|$. Finally, given a finite word $\a\in \cup_{n=1}^{\infty}\A^{n}$ we let $$[\a]:=\left\{(b_n)\in \A^{\mathbb{N}}: b_{1}\ldots b_{|\a|}=\a\right\}.$$ We will often refer to $[\a]$ as the cylinder set corresponding to $\a$.

\subsection{Statement of Theorem \ref{Main theorem}}
The family of limsup sets that will be the main focus of this paper are defined as follows: Given an IFS $\Phi$ and a function $\Psi:\cup_{n=1}^{\infty} \A^{n}\to [0,\infty),$ we let $$W_{\Phi}(\Psi):=\bigcap_{N=1}^{\infty}\bigcup_{n=N}^{\infty}\bigcup_{\a\in \A^n}\bigcup_{l=0}^{n-1}B\left(\pi(a_1\cdots a_{l}(a_{l+1}\cdots a_n)^{\infty}),\Psi(\a)\right).$$ Alternatively, $W_{\Phi}(\Psi)$ is the set of $x\in\mathbb{R}^d$ such that for infinitely many $n$, there exists $\a\in \A^n$ and $0\leq l\leq n-1$ such that $$\|x-\pi(a_1\cdots a_{l}(a_{l+1}\cdots a_n)^{\infty})\|<\Psi(\a).$$ The connection between $W_{\Phi}(\Psi)$ and intrinsic Diophantine Approximation will be made clear in Section \ref{Application section}. Our main result demonstrates that for certain choices of $\Psi,$ the measure of $W_{\Phi}(\Psi)$ is determined by naturally occurring volume sums. One cannot expect such a behaviour to occur for all choices of $\Psi$. Indeed Example 2.1 from \cite{Bak} shows that for a related family of limsup sets, if we want the measure of these limsup sets to be determined by volume sums, then the underlying $\Psi$ should reflect the different rates of scaling within the IFS. As such we will often restrict ourselves to $\Psi$ of the form $$\Psi(\a)=Diam(X_{\a})\cdot g(|\a|)$$ where $g:\mathbb{N}\to[0,\infty)$. This restriction was also adopted in \cite{AllenBar,Bak2,Bakover, HV}. Note that if $\Phi$ is equicontractive, then the set of $\Psi$ that are of this form can be identified with the set of $\Psi$ such that $\Psi(\a)$ only depends upon the length of $\a$.

Our main result is the following statement.
\begin{thm}
	\label{Main theorem}
Let $\Phi=\{\phi_{a}\}_{a\in \A}$ be an IFS and $\Psi:\cup_{n=1}^{\infty}\A^n\to [0,\infty).$ Then the following statements are true:
\begin{enumerate}
	\item For any $s\geq 0$, suppose that $$\sum_{n=1}^{\infty}\sum_{\a\in \A^n}n \cdot \Psi(\a)^{s}<\infty.$$ Then $\mathcal{H}^{s}(W_{\Phi}(\Psi))=0.$
	\item Assume that 
	\begin{equation}
	\label{measure inequality}
	h_{\p}<-2\log \sum_{a\in \A}p_{a}^{2}
	\end{equation} and $\Psi$ is of the form $\Psi(\a)=Diam(X_{\a})g(|\a|)$ for some non-increasing $g:\mathbb{N}\to [0,\infty).$ If $$\sum_{n=1}^{\infty}\sum_{\a\in\A^n}n\cdot (Diam(X_{\a})g(n))^{\dim_{S}(\Phi)}=\infty$$ then $\mu(W_{\Phi}(\Psi))=1.$
	\item Assume that $\Phi$ is equicontractive and $\Psi$ is of the form  $\Psi(\a)=Diam(X_{\a})g(|\a|)$ for some $g:\mathbb{N}\to [0,\infty).$ If $$\sum_{n=1}^{\infty}\sum_{\a\in \A^n}n \cdot (Diam(X_{\a})g(n))^{\dim_{S}(\Phi)}=\infty$$ then $\mu (W_{\Phi}(\Psi))=1.$
\end{enumerate}
\end{thm}
We conclude this section with some remarks on Theorem \ref{Main theorem}.

\begin{remark}
Statement 3 of Theorem \ref{Main theorem} was proved for the IFS $\{\phi_{1}(x)=\frac{x}{3},\phi_{2}(x)=\frac{x+2}{3}\}$ by Tan, Wang, and Wu in \cite{TWW}. Note that this IFS has the middle third Cantor set as its self-similar set. In a recent talk Wang \cite{Wang} commented that the methods used in \cite{TWW} could be generalised to prove Statement 3 of Theorem \ref{Main theorem} for equicontractive IFSs acting on $\mathbb{R}$ that satisfy the strong separation condition. During this talk Wang posed the question as to what happens for IFSs that are not equicontractive. This paper was in part motivated by this question and Statement 2 of Theorem \ref{Main theorem} provides a partial answer. Importantly, as well as providing information in the non-equicontractive case, Theorem \ref{Main theorem} also applies in arbitrary dimensions and requires no separation assumptions on the IFS. The techniques of \cite{TWW} do not apply in this generality. That being said, our method of proof largely follows the same overall strategy as \cite{TWW}. The major differences being that we require additional arguments to control the different rates of scaling within our potentially non-equicontractive IFS, and we also require a new argument to address the potential overlaps that may be present within the IFS. The latter argument uses ideas from \cite{Bakover}.
\end{remark}

\begin{remark}
If $\Phi$ satisfies the open set condition then it is known that $\mu$ is equivalent to the restriction of the $\mathcal{H}^{\dim_{S}(\Phi)}$-dimensional Hausdorff measure on $X$. As such, under the open set condition, Statements 1, 2, and 3 of Theorem \ref{Main theorem} provide a nearly complete description of the $\mu$ measure of $W_{\Phi}(\Psi)$ for $\Psi$ of the form $\Psi(\a)=Diam(X_{\a})g(|\a|)$. Moreover, if we assume that $\Phi$ is equicontractive and satisfies the open set condition, then Statements $1$ and $3$ do provide a complete description. In the overlapping case, i.e. when the open set condition is not satisfied, then Statements 2 and 3 can be used to deduce a number of corollaries on the Hausdorff dimension of $W_{\Phi}(\Psi)$. For if $\dim \mu=\min\{\dim_{S}(\Phi),d\}$ and $\mu(W_{\Phi}(\Psi))=1,$ then we must have $\dim_{H}(W_{\Phi}(\Psi))\geq \min\{\dim_{S}(\Phi),d\}.$ Moreover because $W_{\Phi}(\Psi)$ is a subset of the self similar set $X$, and $X$ satisfies \eqref{similarity inequality}, we must then have $\dim_{H}(W_{\Phi}(X))=\min\{\dim_{S}(\Phi),d\}.$ The important part in this argument is determining when we have $\dim \mu=\min\{\dim_{S}(\Phi),d\}.$ A number of significant breakthroughs on this topic have been made in recent years, see \cite{Hochman,Hochman2,Rap}. These papers provide general sufficient conditions which guarantee $\dim \mu=\min\{\dim_{S}(\Phi),d\}.$ We won't state the results of these papers in their full generality here. Instead we will focus on one particular consequence that is relevant to our purposes. Suppose that $\Phi=\{\phi_{a}(x)=r_{a}x+t_{a}\}$ is an IFS acting on $\mathbb{R}$ and that each $r_{a}$ is algebraic, then it follows from the results of \cite{Rap} that if $\Phi$ does not contain an exact overlap then $\dim \mu=\min\{\dim_{S}(\Phi),d\}.$ We recall that an IFS is said to contain an exact overlap if there exists $\a,\b\in\cup_{n=1}^{\infty}\A^n$ such that $\phi_{\a}=\phi_{\b}$ and $\a\neq \b$. 
\end{remark}

\begin{remark}
The inequality \eqref{measure inequality} and the non-increasing assumption on $g$ in Statement $2$ are both technical assumptions and are believed to be non-optimal. Note that in Statement $3$ there are no monotonicity conditions imposed on $g$. We expect that both of these assumptions can be removed. For the purposes of our exposition, we highlight that in the case of an IFS consisting of two similarities $\{\phi_{1}(x)=r_{1}x+t_{1},\phi_{2}(x)=r_{2}x+t_2\},$ then \eqref{measure inequality} is satisfied if $r_{1}^{\dim_{S}(\Phi)}$ satisfies $$0.048\ldots < r_{1}^{\dim_{S}(\Phi)}< 0.951\ldots.$$ As such we see that \eqref{measure inequality} is satisfied by a significant proportion of those IFSs consisting of two similarities. We remark that when $\Phi$ is an equicontractive IFS then we have $h_{\p}=-\log \sum_{a\in \A}p_{a}^{2},$ and so \eqref{measure inequality} is automatically satisfied in this case. It follows from this observation and the fact that the quantities on each side of \eqref{measure inequality} depend continuously on $\p$, that if we fix the number of maps within our IFS to be $N$ for some $N\in \mathbb{N}$ and identify the space of contraction ratios with $(0,1)^{N}$, then for a non-empty open set of contraction ratios the inequality \eqref{measure inequality} is satisfied. We will see an explicit example where the inequality \eqref{measure inequality} is satisfied in Section \ref{Application section}.
\end{remark}
\begin{remark}
It is a simple exercise to show that a function $g:\mathbb{N}\to [0,\infty)$ satisfies $$\sum_{n=1}^{\infty}\sum_{\a\in \A^n}n \cdot (Diam(X_{\a})g(n))^{\dim_{S}(\Phi)}=\infty$$ if and only if 
$$\sum_{n=1}^{\infty}n \cdot g(n)^{\dim_{S}(\Phi)}=\infty.$$ It will on occasion be more convenient to use this latter divergence condition.
\end{remark}

\noindent \textbf{Notation.} In this paper we will adopt the following notational convention. Given a set $S$ and two functions $f,g:S\to \mathbb{R}$ we write $f\ll g$ if there exists $C>0$ such that $|f(x)|\leq C|g(x)|$ for all $x\in S$. We write $f\asymp g$ if $f\ll g$ and $g \ll f$. 
\section{Applications to intrinsic Diophantine Approximation}
\label{Application section}
\subsection{Background}
The study of intrinsic Diophantine Approximation for self-similar sets has its origins in a question of Mahler \cite{Mah}. He asked how well can elements of the middle third Cantor set $C$ be approximated by rational numbers lying within $C$. To the best of the author's knowledge, the first significant progress in this direction was the work of Levesley, Salp, and Velani \cite{LSV}. They considered rational approximations of the from $p/3^n$. Or equivalently, rational approximations provided by the end points of the sets of the form $\phi_{\a}(C)$. They proved a general Khintchine type result for approximations of this type, see \cite[Theorem 1]{LSV}. Using this theorem, they were able to prove that there exists well-approximable numbers in the middle third Cantor set that are not Liouville. This was an unproved assertion attributed to Mahler. Bugeaud also proved this assertion using a different method in \cite{Bug2}. He in fact provided explicit examples of elements of the middle third Cantor set with any irrationality exponent. In \cite{BugDur} Bugeaud and Durand posed a conjecture on the value of the Hausdorff dimension of the set of points in the middle third Cantor set whose irrationality exponent exceeds a given parameter. Interestingly this conjecture suggests that a phase transition should occur for the value of the Hausdorff dimension of this set. The main result of \cite{BugDur} shows that a version of this conjecture holds almost surely for a particular random model of $C$. The following intrinsic analogue of Dirichlet's theorem for $C$ was proved by Broderick, Fishman, and Reich \cite{BFR}.

\begin{thm}
	\label{BFR theorem}
	For any $x\in C$ and $Q>1$, there exists $p/q\in C$ with $1\leq q\leq Q$ such that $$|x-p/q|<\frac{1}{q(\log_{3}Q)^{\log 3/\log 2}}.$$
\end{thm} Theorem \ref{BFR theorem} was shown to be optimal by Fishman and Simmons \cite{FS}, and Fishman, Merrill, and Simmons \cite{FMS}.

The study of intrinsic Diophantine approximation for self-similar sets naturally leads one to study limsup sets that are in a sense built using the underlying iterated function system. A number of papers have appeared which study such sets, see \cite{AllenBar,Bak2,Bak,Bakerapprox2,Bakover,BakTro,PerRev,PerRevA}. Despite being a problem that was originally motivated by number theoretic considerations, the study of these limsup sets is connected to topics from Ergodic Theory and Fractal Geometry. Interestingly the metric properties of these limsup sets can be related to the absolute continuity of self-similar measures, see \cite{Bakover}.

\subsection{Applications}
For the rest of this section we restrict our attention to iterated function systems of the form $$\Phi=\left\{\phi_{a}(x)=\frac{x+\mathfrak{p}_{a}}{q_{a}}\right\}_{a\in \A}.$$ Where for all $a\in \A$ we have $\pp_{a}\in \mathbb{Z}^{d},$  $q_{a}\in\mathbb{Z},$ and $q_{a}$ also satisfies $|q_{a}|\geq 2.$ For this IFS, the projection map $\pi$ takes the following simplified form:
\begin{equation}
\label{simplified form}
\pi(\a)=\sum_{n=1}^{\infty}\frac{\pp_{a_n}}{\prod_{l=1}^{n}q_{a_{l}}}.
\end{equation} For such an IFS, it is a simple exercise to show that for any $x\in X$ we have that $x\in \mathbb{Q}^{d}$ if and only if there exists $\a\in \cup_{n=1}^{\infty}\A^{n}$ and $0\leq l \leq |\a|-1$ such that $$x=\pi(a_{1}\ldots a_{l}(a_{l+1}\ldots a_{|\a|})^{\infty}).$$ In which case, using properties of geometric series and \eqref{simplified form}, we may conclude that there exists $\pp_{x}\in \mathbb{Z}^d$ such that $$x=\frac{\pp_{x}}{\prod_{j=1}^{l}q_{a_j}\cdot (\prod_{j=l+1}^{n}q_{a_j}-1)}.$$ One of the major difficulties in understanding the properties of the rational numbers within a self-similar set is not knowing if any cancellation occurs between the entries in the vector $\pp_{x}$ and the $\prod_{j=1}^{l}q_{a_j}\cdot (\prod_{j=l+1}^{n}q_{a_j}-1)$ term. It is possible that these two terms contain many common factors and as such $x$ could be written in a significantly reduced form. This makes studying intrinsic Diophantine Approximation for self-similar sets more challenging. For ease of exposition we split what remains of this section into two cases, when $\Phi$ is equicontractive and the general case. Our applications are much simpler to state in the equicontractive case.

\subsubsection{The equicontractive case}
In this section we assume that $\Phi$ is equicontractive, i.e. there exists $q_{\Phi}\in \mathbb{Z}$ such that $q_{a}=q_{\Phi}$ for all $a\in A$. By the above we know that if $\pp/q\in X$ then there exists $\pp'
\in \mathbb{Z}^d,$ $n\in\mathbb{N}$, and $0\leq l\leq n-1$ such that $$\pp/q=\frac{\pp'}{q_{\Phi}^{l}(q_{\Phi}^{n-l}-1)}.$$ We define the intrinsic denominator of $\pp/q\in X$ to be
\begin{align*}
q_{int}(\pp/q):=\inf\Big\{q_{\Phi}^{l}(q_{\Phi}^{|\a|-l}-1):&\,\a\in \cup_{n=1}^{\infty}\A^{n}, 0\leq l\leq |\a|-1
\textrm{ satisfying }\\
&\pp/q=\pi(a_{1}\ldots a_{l}(a_{l+1}\ldots a_{|\a|})^{\infty})\Big\}
\end{align*} 
This is a generalisation of the notion of intrinsic denominator defined in \cite{FS} for $\Phi$ satisfying the strong separation condition. Given a function $\Psi:\mathbb{N}\to[0,\infty)$ we may then define a limsup set as follows:
$$W_{\Phi}^{*}(\Psi):=\left\{x\in X:\|x-\pp/q\|<\Psi(q_{int}(\pp/q))\textrm{ for i.m. }\pp/q\in X\right\}.$$ Fishman and Simmons in \cite{FS} proved a version of Khintchine's theorem for limsup sets of the form $W_{\Phi}^{*}(\Psi)$ when the underlying $\Phi$ is equicontractive, acting on $\mathbb{R},$ and satisfies the strong separation condition. Importantly this result did not provide a complete metric description for the sets $W_{\Phi}^{*}(\Psi),$ even in the restricted case of equicontractive IFSs acting on $\mathbb{R}$ which satisfy the strong separation condition. The divergence condition they needed for a full measure statement was not optimal. This issue was addressed in a recent paper by Tan, Wang, and Wu \cite{TWW} who established a complete analogue of Khintchine's theorem for the set $W_{\Phi}^{*}(\Psi)$ for the IFS $\{\phi_{1}(x)=\frac{x}{3}, \phi_{2}(x)=\frac{x+2}{3}\}.$ Note that this IFS has the middle third Cantor set as its self-similar set. 

\begin{thm}{\cite[Theorem 1.4]{TWW}}
	\label{TWW theorem}
Let $\Phi=\{\phi_{1}(x)=\frac{x}{3}, \phi_{2}(x)=\frac{x+2}{3}\}$ and $\Psi:\mathbb{N}\to [0,\infty)$ be a non-increasing function. Then
\[ \mu(W_{\Phi}^{*}(\Psi)) = \left\{ \begin{array}{ll}
0 & \mbox{if $\sum_{n=1}^{\infty}n\cdot 2^{n}\cdot  \Psi(3^n)^{\frac{\log 2}{\log 3}}<\infty$};\\
1 & \mbox{if $\sum_{n=1}^{\infty}n\cdot 2^{n}\cdot  \Psi(3^n)^{\frac{\log 2}{\log 3}}=\infty$}.\end{array} \right. \] 
\end{thm}
The proof given in \cite{TWW} can be generalised to prove an analogue of Theorem \ref{TWW theorem} for any equicontractive IFS acting on $\mathbb{R}$ satisfying the strong separation condition. Our main result in this direction is the following statement. It generalises Theorem \ref{TWW theorem} to arbitrary dimensions and requires no separation conditions for $\Phi$.
\begin{thm}
	\label{EquiTheorem}
	Let $\Phi$ be an equicontractive IFS of the form $\Phi=\{\phi_{a}(x)=\frac{x+\pp_{a}}{q_{\Phi}}\}.$ Let $\Psi:\mathbb{N}\to [0,\infty)$ be a non-increasing function. Then the following statements are true:
	\begin{enumerate}
		\item For any $s\geq 0$, suppose that $\sum_{n=1}^{\infty}n\cdot \#\A^{n}\cdot  \Psi(q_{\Phi}^n)^{s}<\infty.$ Then $\mathcal{H}^{s}(W_{\Phi}^*(\Psi))=0.$
		\item If $\sum_{n=1}^{\infty}n\cdot \#\A^{n}\cdot  \Psi(q_{\Phi}^n)^{\frac{\log \#\A}{\log q_{\Phi}}}=\infty$ then $\mu(W_{\Phi}^*(\Psi))=1.$
	\end{enumerate}
\end{thm}

\begin{figure}[h]
	
	\begin{subfigure}{0.5\textwidth}
		\includegraphics[width=0.9\linewidth, height=6cm]{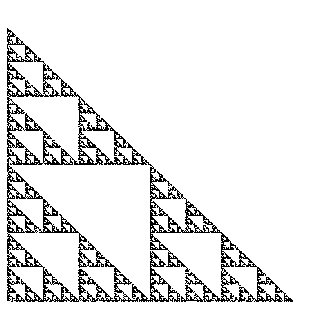} 
		\caption{The self-similar set for $\Phi_{1}.$}
		\label{fig:subim1}
	\end{subfigure}
	\begin{subfigure}{0.5\textwidth}
		\includegraphics[width=0.9\linewidth, height=6cm]{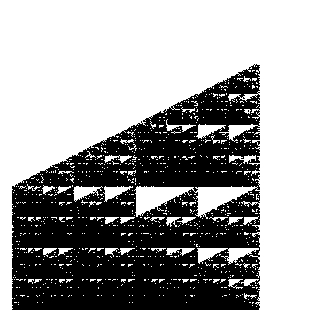}
		\caption{The self-similar set for $\Phi_{2}$.}
		\label{fig:subim2}
	\end{subfigure}
	
	\caption{Let $\Phi_{1}=\{\phi_{1}(x,y)=(\frac{x}{2},\frac{y}{2}),\,\phi_{2}(x,y)=(\frac{x+1}{2},\frac{y}{2}),\phi_{3}(x,y)=(\frac{x}{2},\frac{y+1}{2})\}$ and $\Phi_{2}=\{\phi_{1}(x,y)=(\frac{x}{2},\frac{y}{2})),\, \phi_{2}(x,y)=(\frac{x+1}{2},\frac{y}{2})),\, \phi_{3}(x,y)=(\frac{x}{2},\frac{y+1}{2})),\, \phi_{4}(x,y)=(\frac{x+2}{2},\frac{y+2}{2})),\, \phi_{5}(x,y)=(\frac{x+2}{2},\frac{y}{2}))\}.$ $\Phi_{1}$ satisfies the open set condition and $\Phi_{2}$ does not. Theorem \ref{Main theorem} and Theorem \ref{EquiTheorem} apply to both of these iterated function systems.}
	\label{fig:image2}
\end{figure}

\begin{proof}
We begin by remarking that for any word $\a$ and $0\leq l\leq |\a|-1$ we have \begin{equation}
\label{comparable}
q_{\Phi}^{|\a|-1}\leq q_{\Phi}^{l}(q_{\Phi}^{|\a|-l}-1)\leq q_{\Phi}^{|\a|}.
\end{equation} 

\noindent \textbf{Proof of the convergence case.} To prove the convergence case let $\Psi':\cup_{n=1}^{\infty}\A^{n}\to [0,\infty)$ be given by $$\Psi'(\a)=\Psi(q_{\Phi}^{|\a|-1}).$$ Now notice that for any $\pp/q\in X$, there must exist $\a$ and $0\leq l\leq |\a|-1$ such that 
$$B(\pp/q,\Psi(q_{int}(\pp/q)))=B(\pi(a_{1}\ldots a_{l}(a_{l+1}\ldots a_{|\a|})^{\infty}), \Psi(q_{\Phi}^{l}(q_{\Phi}^{|\a|}-1)).$$ Therefore, using the fact $\Psi$ is non-increasing together with \eqref{comparable} we have 
\begin{align*}
B(\pp/q,\Psi(q_{int}(\pp/q)))&\subseteq B(\pi(a_{1}\ldots a_{l}(a_{l+1}\ldots a_{|\a|})^{\infty}), \Psi(q_{\Phi}^{|\a|-1})\\
&=B(\pi(a_{1}\ldots a_{l}(a_{l+1}\ldots a_{|\a|})^{\infty}), \Psi'(\a)).
\end{align*} Therefore $W_{\Phi}^{*}(\Psi)\subseteq W_{\Phi}(\Psi').$ By Theorem \ref{Main theorem} it will follow that $\mathcal{H}^{s}(W_{\Phi}^{*}(\Psi))=0$ if $\sum_{n=1}^{\infty} \sum_{\a\in \A^n}n\cdot \Psi'(\a)^{s}<\infty.$ This latter inequality is equivalent to $\sum_{n=1}^{\infty}n\cdot \#\A^{n}\cdot \Psi(q_{\Phi}^{n-1})^{s}<\infty.$ However, this inequality is implied by our assumption $\sum_{n=1}^{\infty}n\cdot \#\A^{n}\cdot  \Psi(q_{\Phi}^n)^{s}<\infty.$\\

\noindent \textbf{Proof of the divergence case.} Let $g:\mathbb{N}\to [0,\infty)$ be given by $$g(n)=\frac{\Psi(q_{\Phi}^n)}{Diam(X)\cdot q_{\Phi}^{-n}}.$$ Now define $\Psi':\cup_{n=1}^{\infty}\A^{n}\to [0,\infty)$ by $\Psi'(\a)=Diam(X_{\a})\cdot g(|\a|)$. For any $\a\in \A^n$ and $0\leq l\leq |\a|-1$ there exists $\pp/q\in X$ such that 
$$B(\pi(a_{1}\ldots a_{l}(a_{l+1}\ldots a_{n})^{\infty}),\Psi'(\a))=B(\pp/q,\Psi(q_{\Phi}^{n})).$$ By the non-increasing assumption on $\Psi$, the definition of intrinsic denominator, and \eqref{comparable}, it follows that $$B(\pi(a_{1}\ldots a_{l}(a_{l+1}\ldots a_{n})^{\infty}),\Psi'(\a))\subseteq B(\pp/q,\Psi(q_{int}(\pp/q))).$$ Therefore $W_{\Phi}(\Psi')\subseteq W_{\Phi}^{*}(\Psi)$. By Theorem \ref{Main theorem} it will follow that $\mu(W_{\Phi}^{*}(\Psi))=1$ if we can show that $$\sum_{n=1}^{\infty} \sum_{\a\in \A^n}n\cdot (Diam(X_{\a})g(|\a|))^{\frac{\log \#\A}{\log q_{\Phi}}}=\infty.$$ By the definition of $g$ this is equivalent to our assumption $\sum_{n=1}^{\infty}n\cdot \#\A^{n}\cdot  \Psi(q_{\Phi}^n)^{\frac{\log \#\A}{\log q_{\Phi}}}=\infty.$ Therefore our result follows. 
\end{proof}
	
	
	\label{fig:image2}

If a rational vector $\pp/q\in X$ is in its reduced form then we must have $q\leq q_{int}(\pp/q)$. This observation together with Theorem \ref{EquiTheorem} implies the following statement for traditional rational approximations, where the neighbourhood is defined in terms of the denominator of $\pp/q$ rather than $q_{int}(\pp/q)$.

\begin{cor}
Let $\Phi$ be an equicontractive IFS of the form $\Phi=\{\phi_{\a}(x)=\frac{x+\pp_{a}}{q_{\Phi}}\}.$ Let $\Psi:\mathbb{N}\to [0,\infty)$ be a non-increasing function. If $\sum_{n=1}^{\infty}n\cdot \#\A^{n}\cdot  \Psi(q_{\Phi}^n)^{\frac{\log \#\A}{\log q_{\Phi}}}=\infty$ then 
$$\mu\left(\{x\in X:\|x-\pp/q\|<\Psi(q)\textrm{ for i.m. }\pp/q\in X\}\right)=1.$$
\end{cor}
\subsubsection{The general case}
In this section we no longer assume that $\Phi$ is equicontractive. We formulate two statements, one when $\Phi$ is potentially overlapping, and one when $\Phi$ satisfies the strong separation condition. 

Given $g:\mathbb{N}\to [0,\infty)$ and $\pp/q\in X$ we define
\begin{align*}
\Psi_{g}(\pp/q):=\max\Big\{ Diam(X_{\a})\cdot g(|\a|):&\, \a\in\cup_{n=1}^{\infty} \A^n, 0\leq l\leq |\a|-1 \textrm{ satisfying}\\
&\pp/q=\pi(a_{1}\ldots a_{l}(a_{l+1}\ldots a_{|\a|})^{\infty})
\Big\}.
\end{align*} 
To each $g:\mathbb{N}\to [0,\infty)$ we associate the set
$$W_{\Phi}^{**}(g)=\{x\in X:\|x-\pp/q\|<\Psi_{g}(\pp/q)\textrm{ for i.m. }\pp/q\in X\}.$$
The following statement is essentially Theorem \ref{Main theorem} rephrased in terms of rational approximations.
\begin{thm}
Let $\Phi$ be an IFS of the form $\Phi=\{\phi_{\a}(x)=\frac{x+\pp_{a}}{q_{a}}\}$ and $g:\mathbb{N}\to [0,\infty)$. Then the following statement are true
\begin{enumerate}
	\item For any $s\geq 0$, suppose that $\sum_{n=1}^{\infty}\sum_{\a \in \A^{n}}n\cdot (Diam(X_{\a})g(n))^{s}<\infty.$ Then $\mathcal{H}^{s}(W_{\Phi}^{**}(g))=0.$
	\item If $g$ is non-increasing, $h_{\p}<-2\log \sum_{a\in \A}p_{a}^{2},$ and $\sum_{n=1}^{\infty}\sum_{\a \in \A^{n}}n\cdot (Diam(X_{\a})g(n))^{\dim_{S}(\Phi)}=\infty$ then $\mu(W_{\Phi}^{**}(g))=1.$
\end{enumerate}
\end{thm}
\begin{proof}
This result follows from Theorem \ref{Main theorem} together with the observation that if $x\notin \mathbb{Q}^d$ then $x\in W_{\Phi}^{**}(\Psi_{g})$ if and only if $x\in W_{\Phi}(\Psi)$ for $\Psi(\a)=Diam(X_{\a})g(|\a|).$
\end{proof}
For what remains of this section we will always assume that $\Phi$ satisfies the strong separation condition. Because of the strong separation condition, for any $\pp/q\in X$ there exists a unique sequence $(a_n)\in \A^{\mathbb{N}}$ satisfying $\pi((a_n))=\pp/q$. We emphasise that $(a_n)$ must be eventually periodic. Given $\pp/q\in X$ we define the intrinsic denominator of $\pp/q$ to be $$q_{int}(\pp/q):=\inf\left\{\prod_{j=1}^{l}q_{a_j}\cdot (\prod_{j=l+1}^{n}q_{a_j}-1): n\in\mathbb{N}, 0\leq l\leq n-1, \textrm{ and } \pp/q=\pi(a_{1}\ldots a_{l}(a_{l+1}\ldots a_{n})^{\infty})\right\}.$$ For any $\pp/q\in X$ we define $$n(\pp/q):=\inf\{n\in \mathbb{N}:\pp/q=\pi(a_{1}\ldots a_{l}(a_{l+1}\ldots a_{n})^{\infty})\textrm{ for some }0\leq l\leq n-1\}$$ Similarly, we let $$l(\pp/q):=\inf\{0\leq l\leq n(\pp/q)-1:\pp/q=\pi(a_{1}\ldots a_{l}(a_{l+1}\ldots a_{n_{p/q}})^{\infty})\}.$$ It can be shown that $n(\pp/q)$ and $l(\pp/q)$ are the unique parameters satisfying
$$q_{int}(\pp/q)=\prod_{j=1}^{l(\pp/q)}q_{a_j}\cdot (\prod_{j=l(\pp/q)+1}^{n(\pp/q)}q_{a_j}-1)$$ and $$\pp/q=\pi(a_{1}\ldots a_{l(\pp/q)}(a_{l(\pp/q)+1}\ldots a_{n(\pp/q)})^{\infty}).$$ Given $g:\mathbb{N}\to [0,\infty)$ we define a limsup set as follows, let $$W_{\Phi}^{***}(g):=\left\{x\in X:\|x-\pp/q\|<\frac{g(n(\pp/q))}{q_{int}(\pp/q)}\textrm{ for i.m. }\pp/q\in X\right\}.$$


\begin{figure}
	\begin{center}
	\includegraphics[width=250pt]{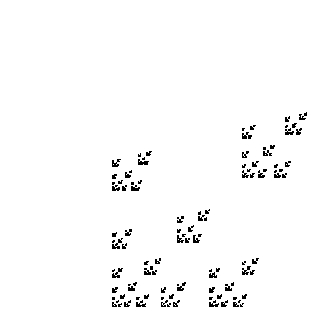}
	\end{center}
	\caption{Let $\Phi=\{\phi_{1}(x,y)=(\frac{x}{2},\frac{y}{2}),\, \phi_{2}(x,y)=(\frac{x+2}{4},\frac{y}{4}),\, \phi_{3}(x,y)=(\frac{x}{5},\frac{y+3}{5}),\, \phi_{4}(x,y)=(\frac{x+2}{3},\frac{y+2}{3})\}.$ The self-similar set for this IFS is displayed above. One can check that this IFS satisfies the strong separation condition and $h_{\p}<-2\log \sum_{a\in \A}p_{a}^2$. Therefore Theorem \ref{Main theorem} and Theorem \ref{General theorem} apply to this IFS.}
	\label{Fig1}
\end{figure}



\begin{thm}
	\label{General theorem}
	Let $\Phi$ be an IFS of the form $\Phi=\{\phi_{a}(x)=\frac{x+\pp_a}{q_a}\}$ satisfying the strong separation condition and let $g:\mathbb{N}\to [0,\infty)$. Then the following statement are true:
	\begin{enumerate}
		\item For any $s\geq 0$, suppose that $\sum_{n=1}^{\infty}\sum_{\a\in \A^{n}}n\cdot (Diam(X_{\a})g(n))^{s}<\infty.$ Then $\mathcal{H}^{s}(W_{\Phi}^{***}(g))=0$.
		\item If $g$ is non-increasing, $h_{\p}<-2\log \sum_{a\in \A}p_{a}^{2},$ and $\sum_{n=1}^{\infty}\sum_{\a\in \A^{n}}n\cdot (Diam(X_{\a})g(n))^{\dim_{S}(\Phi)}=\infty,$ then $\mu(W_{\Phi}^{***}(g))=1.$
	\end{enumerate}  
\end{thm}
\begin{proof}
We begin our proof by remarking that there exists constant $C_{1},C_{2}>0$ such that for any word $\a\in \cup_{n=1}^{\infty}\A^{n}$ and $0\leq l\leq |\a|-1$ we have 
\begin{equation}
\label{Comparable2}
\frac{C_{1}}{\prod_{j=1}^{l}q_{a_j}\cdot (\prod_{j=l+1}^{|\a|}q_{a_j}-1)}< Diam(X_{\a})< \frac{C_{2}}{\prod_{j=1}^{l}q_{a_j}\cdot (\prod_{j=l+1}^{|\a|}q_{a_j}-1)}.
\end{equation}\\

\noindent \textbf{The convergence case.} It follows from \eqref{Comparable2} that for any $\pp/q\in X$ we have $$B\left(\frac{\pp}{q},\frac{g(n(\pp/q))}{q_{int}(\pp/q)}\right)\subseteq B\left(\pi(a_{1}\ldots a_{l(\pp/q)}(a_{l(\pp/q)+1}\ldots a_{n(\pp/q)})^{\infty}),\frac{g(n(\pp/q))Diam(X_{a_{1}\ldots a_{n(\pp/q)}})}{C_{1}}\right).$$ Therefore $W_{\Phi}^{***}(g)\subseteq W_{\Phi}(\Psi)$ for $\Psi(\a)=\frac{g(|\a|)Diam(X_{\a})}{C_{1}}.$ In which case $\mathcal{H}^{s}(W_{\Phi}^{***}(g))=0$ if $\mathcal{H}^{s}(W_{\Phi}(\Psi))=0.$ However this last inequality follows from Theorem \ref{Main theorem} and our assumption $\sum_{n=1}^{\infty}\sum_{\a\in \A^{n}}n\cdot (Diam(X_{\a})g(n))^{s}<\infty.$\\

\noindent \textbf{The divergence case.} 
For any $\a\in \cup_{n=1}^{\infty}\A^n$ and $0\leq l\leq |\a|-1$ there exists $\pp/q\in X$ such that $\pp/q=\pi(a_{1}\ldots a_{l}(a_{l+1}\ldots a_{|\a|})^{\infty}).$ Therefore
$$B\left(\pi(a_{1}\ldots a_{l}(a_{l+1}\ldots a_{|\a|})^{\infty}),\frac{g(|\a|)Diam(X_{\a})}{C_{2}}\right)=B\left(\pp/q,\frac{g(|\a|)Diam(X_{\a})}{C_{2}}\right).$$ It now follows from the fact $g$ is non-increasing together with \eqref{Comparable2} that 
\begin{align*}
&B\left(\pi(a_{1}\ldots a_{l}(a_{l+1}\ldots a_{|\a|})^{\infty}),\frac{g(|\a|)Diam(X_{\a})}{C_{2}}\right)\\
\subseteq &B\left(\pp/q,\frac{g(n(\pp/q))Diam(X_{a_{1}\ldots a_{n(\pp/q)}})}{C_{2}}\right)\\
\subseteq &B\left(\pp/q,\frac{g(n(\pp/q))}{q_{int}(\pp/q)}\right)
\end{align*}
Therefore $W_{\Phi}(\Psi)\subset W_{\Phi}^{***}(g)$ for $\Psi$ given by $\Psi(\a)=\frac{g(|\a|)Diam(X_{\a})}{C_{2}}$, and our result follows if $\mu(W_{\Phi}(\Psi))=1.$ However $\mu(W_{\Phi}(\Psi))=1$ follows from Theorem \ref{Main theorem} together with our assumption $\sum_{n=1}^{\infty}\sum_{\a\in \A^{n}}n\cdot (Diam(X_{\a})g(n))^{\dim_{S}(\Phi)}=\infty.$
\end{proof}
As remarked upon in the equicontractive case, if $\pp/q\in X$ is in its reduced from then we must have $q\leq q_{int}(\pp/q)$. This observation together with Theorem \ref{General theorem} implies the following corollary.
\begin{cor}
Let $\Phi$ be an IFS of the form $\Phi=\{\phi_{a}(x)=\frac{x+\p_a}{q_a}\}$ satisfying the strong separation condition and $h_{\p}<-2\log \sum_{a\in \A}p_{a}^{2}$. Let $g:\mathbb{N}\to [0,\infty)$ be a non-increasing function. If $\sum_{n=1}^{\infty}\sum_{\a\in \A^{n}}n\cdot (Diam(X_{\a})g(n))^{\dim_{S}(\Phi)}=\infty$ then 
$$\mu\left(\left\{x\in X:\|x-\pp/q\|<\frac{g(n(\pp/q))}{q}\textrm{ for i.m. }\pp/q\in X\right\}\right)=1.$$
\end{cor}

\section{Proof of Theorem \ref{Main theorem}}
\label{Proof section}
In this section we prove Theorem \ref{Main theorem}. Statement $1$ of this theorem is proved by a standard covering argument and as such is omitted. We will only prove Statement $2$ of Theorem \ref{Main theorem} in full. Statement $3$ is proved via an almost identical method. Where appropriate we will indicate in the footnotes where the proofs differ and why Statement $3$ does not require the assumption $g$ is non-increasing. 
\subsection{Technical preliminaries}

Given a word $\a\in \cup_{n=1}^{\infty}\A^n$ we let $C_{t}(\a)$ denote the number of distinct words of length $t$ appearing within $\a$. Given a probability vector $\p$ and $n\in\mathbb{N}$ we let $$k_{n,\p}:=\left \lfloor\frac{-\log n}{\log \sum_{a\in \A}p_{a}^2}\right\rfloor +1.$$ We also let $$\mathcal{F}_{n,\p}:=\left\{\a\in \A^{n}:C_{k_{n,\p}}(\a)\geq \Big\lfloor \frac{n}{10}\Big\rfloor \right\}.$$ When the choice of $\p$ is implicit, we simply denote $k_{n,\p}$ by $k_{n}$ and $\mathcal{F}_{n,\p}$ by $\mathcal{F}_{n}.$

The following lemma is a suitable adaptation of Lemma 4.1. from \cite{TWW}.  

\begin{lemma}
	\label{subwords lemma}
	Let $\p$ be a probability vector and $\m$ be the corresponding Bernoulli measure on $\A^{\mathbb{N}}$. Then for $n$ sufficiently large, we have $$\m\left(\bigcup_{\a\in \mathcal{F}_{n}}[\a]\right)\geq \frac{7}{32}.$$
\end{lemma}

\begin{proof}
Given $\a\in \A^{n}$ and $\b \in \A^{k_n},$ let $$|\a|_{\b}:=\# \{0\leq l\leq n-k_{n}:a_{l+1}\ldots a_{l+k_{n}}=\b\}.$$ 
It is convenient to express our proof using the language of probability theory. As such let $(Z_l)_{l=1}^n$ be a sequence of i.i.d random variables taking values in $\A$ such that $\mathbb{P}(Z_l=a)=p_{a}$ for all $a\in \A$. Let $\Sigma_{k_n}^{n}$ be the real valued random variable given by $$\Sigma_{k_n}^{n}:=\sum_{\b\in \A^{k_n}}|Z_{1}\ldots Z_{n}|_{\b}^{2}.$$ We start by bounding the expectation of $\Sigma_{k_n}^{n}$:
\begin{align}
\label{expectation rewrite}
\mathbb{E}(\Sigma_{k_n}^{n})&=\sum_{\b\in \A^{k_n}}\mathbb{E}\left(|Z_{1}\ldots Z_{n}|_{\b}^{2}\right)\nonumber\\
&=\sum_{\b\in \A^{k_n}} \mathbb{E}\left(\left(\sum_{l=0}^{n-k_n}\mathbbm{1}_{\b}(Z_{l+1}\ldots Z_{l+k_n})\right)^{2}\right)\nonumber\\
&=\sum_{\b\in \A^{k_n}}\mathbb{E}\left(\sum_{l,j=0}^{n-k_{n}}\mathbbm{1}_{\b}(Z_{l+1}\ldots Z_{l+k_n})\cdot \mathbbm{1}_{\b}(Z_{j+1}\ldots Z_{j+k_n}))\right)\nonumber\\
&=\sum_{l,j=0}^{n-k_{n}}\mathbb{E}\left(\sum_{\b\in \A^{k_n}}\mathbbm{1}_{\b}(Z_{l+1}\ldots Z_{l+k_n})\cdot \mathbbm{1}_{\b}(Z_{j+1}\ldots Z_{j+k_n}))\right)\nonumber\\
&=\sum_{l,j=0}^{n-k_{n}}\mathbb{E}(\mathbbm{1}_{Z_{l+1}\ldots Z_{l+k_n}=Z_{j+1}\ldots Z_{j+k_{n}}})\nonumber\\
&=\sum_{l,j=0}^{n-k_{n}}\mathbb{P}(Z_{l+1}\ldots Z_{l+k_n}=Z_{j+1}\ldots Z_{j+k_{n}}).
\end{align}
With \eqref{expectation rewrite} in mind, we now bound $\mathbb{P}(Z_{l+1}\ldots Z_{l+k_n}=Z_{j+1}\ldots Z_{j+k_{n}})$ from above. If $l=j$ then clearly $\mathbb{P}(Z_{l+1}\ldots Z_{l+k_n}=Z_{j+1}\ldots Z_{j+k_{n}})=1.$ We remark that for any parameters $l,j,k\in\mathbb{N}$, if $l+k<j+1$ then by independence we have $$\mathbb{P}(Z_{l+1}\ldots Z_{l+k}=Z_{j+1}\ldots Z_{j+k})=\prod_{m=1}^{k}\mathbb{P}(Z_{l+m}=Z_{j+m})=\prod_{m=1}^{k}\left(\sum_{a\in \A}p_{a}^2\right)=\left(\sum_{a\in \A}p_{a}^2\right)^{k}.$$ We will use the fact that if $l+k<j+1$ then 
\begin{equation}
\label{independence identity}
\mathbb{P}(Z_{l+1}\ldots Z_{l+k}=Z_{j+1}\ldots Z_{j+k})=\left(\sum_{a\in \A}p_{a}^2\right)^{k}
\end{equation}throughout our proof.

We now proceed via a case analysis. If $l+k_{n}<j+1$ then \eqref{independence identity} immediately implies 
\begin{equation}
\label{case11}
\mathbb{P}(Z_{l+1}\ldots Z_{l+k_n}=Z_{j+1}\ldots Z_{j+k_n})=\left(\sum_{a\in \A}p_{a}^2\right)^{k_n}.
\end{equation}
Now suppose that $l+\frac{k_n}{4}\leq j+1\leq l+k_{n}$. Observe that $Z_{l+1}\ldots Z_{l+k_n}=Z_{j+1}\ldots Z_{j+k_{n}}$ implies that $Z_{2l-j+k_n+1}\ldots Z_{l+k_n}=Z_{l+k_{n}+1}\ldots Z_{j+k_{n}}.$ Therefore by \eqref{independence identity} we have
\begin{align}
\label{case22}
\mathbb{P}(Z_{l+1}\ldots Z_{l+k_n}=Z_{j+1}\ldots Z_{j+k_{n}})&\leq \mathbb{P}(Z_{2l-j+k_n+1}\ldots Z_{l+k_n}=Z_{l+k_{n}+1}\ldots Z_{j+k_{n}})\nonumber \\
&=\left(\sum_{a\in \A}p_{a}^2\right)^{j-l+1}\nonumber\\
&\leq \left(\sum_{a\in \A}p_{a}^2\right)^{\frac{k_{n}}{4}}.
\end{align} Let us now suppose that $l+1<j+1< l+\frac{k_{n}}{4}$ and $Z_{l+1}\ldots Z_{l+k_n}=Z_{j+1}\ldots Z_{j+k_{n}}.$  Notice that $Z_{l+1}\ldots Z_{l+k_n}=Z_{j+1}\ldots Z_{j+k_{n}}$ implies that $Z_{l+i+1}\ldots Z_{l+i+\lfloor \frac{k_n}{4}\rfloor}=Z_{j+i+1}\ldots Z_{j+i+\lfloor \frac{k_n}{4}\rfloor}$ for any $0\leq i<k_{n}-\lfloor \frac{k_n}{4}\rfloor.$ Repeatedly applying this identity, it follows that \begin{align*}
Z_{l+1}\ldots Z_{l+\lfloor \frac{k_n}{4}\rfloor}=Z_{j+1}\ldots Z_{j+\lfloor \frac{k_n}{4}\rfloor}&=Z_{l+2(j-l)+1}\ldots Z_{l+2(j-l)+\lfloor \frac{k_n}{4}\rfloor}\\
&=\cdots\\
&=Z_{l+d(j-l)+1}\ldots Z_{l+d(j-l)+\lfloor \frac{k_n}{4}\rfloor}
\end{align*} for any $d$ such that $l+d(j-l)+1<k_{n}-\lfloor\frac{k_n}{4}\rfloor.$ Since $j-l<\frac{k_n}{4}$, it follows that we can pick $d$ such that $l+\lfloor \frac{k_n}{4}\rfloor<l+d(j-l)+1\leq k_{n}-\lfloor \frac{k_n}{4}\rfloor.$ Taking such a $d$, it then follows from \eqref{independence identity} that
\begin{align}
\label{case33}
\mathbb{P}(Z_{l+1}\ldots Z_{l+k_n}=Z_{j+1}\ldots Z_{j+k_{n}})&\leq \mathbb{P}(Z_{l+1}\ldots Z_{l+\lfloor \frac{k_n}{4}\rfloor}=Z_{l+d(j-l)+1}\ldots Z_{l+d(j-l)+\lfloor \frac{k_n}{4}\rfloor})\nonumber\\
&=\left(\sum_{a\in \A}p_{a}^2\right)^{\lfloor \frac{k_n}{4}\rfloor}.
\end{align}
Recalling \eqref{expectation rewrite}, we have
\begin{align*}
\mathbb{E}(\Sigma_{k_n}^{n})&=\sum_{l,j=0}^{n-k_{n}}\mathbb{P}(Z_{l+1}\ldots Z_{l+k_n}=Z_{j+1}\ldots Z_{j+k_{n}})\\
&=\sum_{l=0}^{n-k_n}1+2\sum_{l=0}^{n-k_{n}-1}\sum_{j=l+1}^{n-k_n}\mathbb{P}(Z_{l+1}\ldots Z_{l+k_n}=Z_{j+1}\ldots Z_{j+k_{n}})\\
&=\sum_{l=0}^{n-k_n}1+2\sum_{l=0}^{n-k_{n}-1}\left(\sum_{l+k_n< j+1\leq n-k_n}\mathbb{P}(\cdot)+\sum_{l+\frac{k_n}{4}\leq j+1\leq l+k_{n}}\mathbb{P}(\cdot)+\sum_{l+1< j+1<l+\frac{k_n}{4}}\mathbb{P}(\cdot)\right).
\end{align*}Applying the bounds provided by \eqref{case11}, \eqref{case22}, \eqref{case33} we have 
\begin{align*}
\mathbb{E}(\tilde{\Sigma}_{k_n}^{n})&\leq n+2\sum_{l=0}^{n-k_{n}-1}\left(\sum_{l+k_n< j+1\leq n-k_n}\left(\sum_{a\in \A}p_{a}^2\right)^{k_n}+\sum_{l+\frac{k_n}{4}\leq j+1\leq l+k_{n}}\left(\sum_{a\in \A}p_{a}^2\right)^{ \frac{k_n}{4}}+\sum_{l+1< j+1<l+\frac{k_n}{4}}\left(\sum_{a\in \A}p_{a}^2\right)^{\lfloor \frac{k_n}{4}\rfloor}\right)\\
&\leq n+2\sum_{l=0}^{n-k_{n}-1}\left(n\left(\sum_{a\in \A}p_{a}^2\right)^{k_n}+k_{n}\left(\sum_{a\in \A}p_{a}^2\right)^{\lfloor \frac{k_n}{4}\rfloor}\right).
\end{align*}
By the definition of $k_{n}$ we know that $$n\left(\sum_{a\in \A}p_{a}^2\right)^{k_n}\leq 1.$$ Moreover, as $k_{n}$ grows logarithmically in $n$, and $\left(\sum_{a\in \A}p_{a}^2\right)^{\lfloor \frac{k_n}{4}\rfloor}$ decays to zero polynomially fast, we know that  $$k_{n}\left(\sum_{a\in \A}p_{a}^2\right)^{\lfloor \frac{k_n}{4}\rfloor}\leq 1$$ for all $n$ sufficiently large. Therefore for $n$ sufficiently large we have 
\begin{equation}
\label{expectation bound}
\mathbb{E}(\Sigma_{k_n}^{n})\leq n+2\sum_{l=0}^{n-k_{n}-1}2\leq 5n.
\end{equation} Let $B$ be the event that
$$\#\left\{\b\in \A^{k_n}:|Z_{1}\ldots Z_{n}|_{\b}\geq 1\right\}< \Big\lfloor \frac{n}{10}\Big\rfloor .$$ We now bound the probability of $B$ from above. It follows from \eqref{expectation bound} that for $n$ sufficiently large we have
\begin{align*}
5n&\geq \mathbb{E}(\Sigma_{k_n}^n)\\
&\geq \mathbb{E}\left(\sum_{\b\in \A^{k_n}}|Z_{1}\ldots Z_{n}|_{\b}^{2}\Big |B\right)\mathbb{P}(B)\\
&\geq \mathbb{P}(B)\cdot \min\left\{m_{1}^2+\cdots + m_{\lfloor \frac{n}{10}\rfloor}^2:m_{1}+\cdots +m_{\lfloor \frac{n}{10}\rfloor}=n-k_n+1\right\}\\
&=\mathbb{P}(B)\cdot \left(\frac{n-k_n+1}{\lfloor \frac{n}{10}\rfloor}\right)^2\cdot \Big\lfloor \frac{n}{10}\Big\rfloor\\
&=\mathbb{P}(B)\cdot \left(n-k_n+1\right)^2\cdot \Big\lfloor \frac{n}{10}\Big\rfloor^{-1}\\
&\geq \mathbb{P}(B)\cdot 10\cdot\frac{(4n/5)^2}{n}\\
&=\mathbb{P}(B)\cdot \frac{32n}{5}.
\end{align*}In the penultimate line we used that for $n$ sufficiently large we have $n-k_{n}+1\geq \frac{4n}{5}$. This is because $k_n$ grows logarithmically in $n$. Therefore $$\mathbb{P}(B)\leq \frac{25}{32} .$$ This means that $\mathbb{P}(B^c)\geq 7/32$. Since $$\m\left(\bigcup_{\a\in \mathcal{F}_n}[\a]\right)=\mathbb{P}(B^c)$$ this completes our proof. 
\end{proof}
Let us now suppose that we are given a probability vector $\p$ and an IFS $\Phi$. Recall that the entropy of $\p$ and the Lyapunov exponent of $\p$ are defined to be $$h_{\p}=-\sum_{a\in \A}p_{a}\log p_{a} \textrm{ and }\chi_{\Phi,p}=-\sum_{a\in \A}p_{a}\log r_{a}$$ respectively. Given a word $\a\in \A^n$ and $\epsilon>0$ we let 
\begin{align*}
Bad(\a,\epsilon):=&\left\{0\leq l\leq n-k_{n}: \prod_{i=1}^{k_{n}}r_{a_{l+i}}\notin \left[e^{k_n(-\chi_{\Phi,\p}-\epsilon)},e^{k_n(-\chi_{\Phi,\p}+\epsilon)}\right]\right\}\\
\cup&\left\{0\leq l\leq n-k_{n}:\prod_{i=1}^{k_{n}}p_{a_{l+i}}\notin \left[e^{k_n(-h_{\p}-\epsilon)},e^{k_n(-h_{\p}+\epsilon)}\right]\right\}.
\end{align*} We then define $$Bad(n,\epsilon):=\left\{\a\in \A^n: \# Bad(\a,\epsilon)\geq \Big\lfloor \frac{n}{20}\Big\rfloor\right\}.$$

\begin{lemma}
	\label{Lyapunov lemma}
Let $\Phi$ be an IFS, $\p$ be a probability vector, and $\m$ be the Bernoulli measure corresponding to $\p$. For any $\epsilon>0,$ there exists $\gamma\in(0,1)$ such that $$\m\left(\bigcup_{\a\in Bad(n,\epsilon)}[\a]\right)\ll \gamma^{k_n} .$$
\end{lemma}

\begin{proof}
	As in the proof of the previous lemma, it is useful to express this proof in terms of random variables. Let $(Z_{l})_{l=1}^n$ be a sequence of i.i.d. random variables taking values in $\A$ such that $\mathbb{P}(Z_l=a)=p_{a},$ and let $\epsilon>0$ be arbitrary. We start our proof by bounding from above the expectation $$\mathbb{E}\left(\sum_{l=0}^{n-k_n}\mathbbm{1}_{[e^{k_n(-\chi_{\Phi,\p}-\epsilon)},e^{k_n(-\chi_{\Phi,\p}+\epsilon)}]^c}\left(\prod_{i=1}^{k_n}r_{Z_{l+i}}\right)\right).$$
By the linearity of expectation we have 
\begin{align*}\mathbb{E}\left(\sum_{l=0}^{n-k_n}\mathbbm{1}_{[e^{k_n(-\chi_{\Phi,\p}-\epsilon)},e^{k_n(-\chi_{\Phi,\p}+\epsilon)}]^c}\left(\prod_{i=1}^{k_n}r_{Z_{l+i}}\right)\right)&=\sum_{l=0}^{n-k_n}\mathbb{E}\left(\mathbbm{1}_{[e^{k_n(-\chi_{\Phi,\p}-\epsilon)},e^{k_n(-\chi_{\Phi,\p}+\epsilon)}]^c}\left(\prod_{i=1}^{k_n}r_{Z_{l+i}}\right)\right)\\
&=\sum_{l=0}^{n-k_n}\mathbb{P}\left(\prod_{i=1}^{k_n}r_{Z_{l+i}}\notin [e^{k_n(-\chi_{\Phi,\p}-\epsilon)},e^{k_n(-\chi_{\Phi,\p}+\epsilon)}]\right).
\end{align*}
By Hoeffding's inequality for large deviations \cite{Hoe}, there exists $\gamma_1:=\gamma_1(\epsilon,\p,\Phi)\in(0,1)$ such that $$\mathbb{P}\left(\prod_{i=1}^{k_n}r_{Z_{l+i}}\notin [e^{k_n(-\chi_{\Phi,\p}-\epsilon)},e^{k_n(-\chi_{\Phi,\p}+\epsilon)}]\right)\ll \gamma_{1}^{k_n}.$$
Therefore $$\mathbb{E}\left(\sum_{l=0}^{n-k_n}\mathbbm{1}_{[e^{k_n(-\chi_{\Phi,\p}-\epsilon)},e^{k_n(-\chi_{\Phi,\p}+\epsilon)}]^c}\left(\prod_{i=1}^{k_n}r_{Z_{l+i}}\right)\right)\ll n\cdot \gamma_{1}^{k_n}.$$ Now by Markov's inequality, we have $$\mathbb{P}\left(\sum_{l=0}^{n-k_n}\mathbbm{1}_{[e^{k_n(-\chi_{\Phi,\p}-\epsilon)},e^{k_n(-\chi_{\Phi,\p}+\epsilon)}]^{c}}\left(\prod_{i=1}^{k_n}r_{Z_{l+i}}\right)\geq \frac{1}{2}\Big\lfloor\frac{n}{20}\Big\rfloor \right)\ll \gamma_{1}^{k_n}.$$ By an analogous argument, it can be shown that there exists $\gamma_{2}\in (0,1)$ such that $$\mathbb{P}\left(\sum_{l=0}^{n-k_n}\mathbbm{1}_{[e^{k_n(-h_{\p}-\epsilon)},e^{k_n(-h{\p}+\epsilon)}]^{c}}\left(\prod_{i=1}^{k_n}p_{Z_{l+i}}\right)\geq \frac{1}{2}\Big\lfloor\frac{n}{20}\Big\rfloor \right)\ll \gamma_{2}^{k_n}.$$
Now let $\gamma=\max\{\gamma_{1},\gamma_{2}\}$. Clearly if $(Z_l)_{l=1}^n$ is such that either $$\prod_{i=1}^{k_n}r_{Z_{l+i}}\notin [e^{k_n(-\chi_{\Phi,\p}-\epsilon)},e^{k_n(-\chi_{\Phi,\p}+\epsilon)}]\textrm{ or } \prod_{i=1}^{k_{n}}p_{Z_{l+i}}\notin \left[e^{k_n(-h_{\p}-\epsilon)},e^{k_n(-h_{\p}+\epsilon)}\right]$$ for $\lfloor\frac{n}{20}\rfloor$ values of $l,$ it must satisfy $$\prod_{i=1}^{k_{n}}r_{a_{l+i}}\notin \left[e^{k_n(-\chi_{\Phi,\p}-\epsilon)},e^{k_n(-\chi_{\Phi,\p}+\epsilon)}\right]$$ for at least $\frac{1}{2}\lfloor\frac{n}{20}\rfloor$ values of $l$ or $$\prod_{i=1}^{k_{n}}p_{a_{l+i}}\notin \left[e^{k_n(-h_{\p}-\epsilon)},e^{k_n(-h_{\p}+\epsilon)}\right]$$ for at least $\frac{1}{2}\lfloor\frac{n}{20}\rfloor$ values of $l.$ As such we may conclude that
\begin{align*}
&\m\left(\bigcup_{\a\in Bad(n,\epsilon)}[\a]\right)\\
=&\mathbb{P}\left(\#\left\{l:\prod_{i=1}^{k_n}r_{Z_{l+i}}\notin [e^{k_n(-\chi_{\Phi,\p}-\epsilon)},e^{k_n(-\chi_{\Phi,\p}+\epsilon)}] \textrm{ or }\prod_{i=1}^{k_{n}}p_{Z_{l+i}}\notin \left[e^{k_n(-h_{\p}-\epsilon)},e^{k_n(-h_{\p}+\epsilon)}\right]\right\}\geq \Big\lfloor \frac{n}{20}\Big\rfloor\right)\\
\leq& \mathbb{P}\left(\sum_{l=0}^{n-k_n}\mathbbm{1}_{[e^{k_n(-\chi_{\Phi,\p}-\epsilon)},e^{k_n(-\chi_{\Phi,\p}+\epsilon)}]^{c}}\left(\prod_{i=1}^{k_n}r_{a_{l+i}}\right)\geq \frac{1}{2}\Big\lfloor\frac{n}{20}\Big\rfloor \right)\\
&+\mathbb{P}\left(\sum_{l=0}^{n-k_n}\mathbbm{1}_{[e^{k_n(-h_{\p}-\epsilon)},e^{k_n(-h{\p}+\epsilon)}]^{c}}\left(\prod_{i=1}^{k_n}p_{a_{l+i}}\right)\geq \frac{1}{2}\Big\lfloor\frac{n}{20}\Big\rfloor \right)\\
\ll &\gamma^{k_n}.
\end{align*}
This completes our proof.
\end{proof}
Combining Lemmas \ref{subwords lemma} and \ref{Lyapunov lemma}, we see that the following statement holds.

\begin{lemma}
	\label{Important lemma}
	Let $\Phi$ be an IFS, $\p$ be a probability vector, and $\m$ be the Bernoulli measure corresponding to $\p$. For any $\epsilon>0$, it is the case that for all $n$ sufficiently large there exists a set $Good(n,\epsilon)\subset \A^n$ satisfying:
\begin{enumerate}
	\item $\m\left(\bigcup_{\a\in Good(n,\epsilon)}[\a]\right)\geq 7/64.$
	\item For each $\a\in Good(n,\epsilon)$ there exists a set $W_{\a}\subset \{0,\ldots, n-k_{n}\}$ satisfying:
\begin{enumerate}
	\item $\#W_{\a}\geq \Big\lfloor \frac{n}{20}\Big\rfloor$.
\item If $l,l'\in W_{\a}$ and $l\neq l'$ then $a_{l+1}\ldots a_{l+k_{n}}\neq a_{l'+1}\ldots a_{l'+k_{n}}$.
\item For each $l\in W_{\a}$ we have $$\prod_{i=1}^{k_n}r_{a_{l+i}}\in \left[e^{k_n(-\chi_{\Phi,\p}-\epsilon)},e^{k_n(-\chi_{\Phi,\p}+\epsilon)}\right]$$ and $$\prod_{i=1}^{k_n}p_{a_{l+i}}\in  \left[e^{k_n(-h_{\p}-\epsilon)},e^{k_n(-h_{\p}+\epsilon)}\right].$$
\end{enumerate}
\end{enumerate}
\end{lemma}
Lemma \ref{Important lemma} is the main technical result in this section and will play an important part in our proof of Theorem \ref{Main theorem}\footnote{In the proof of Statement $3$ from Theorem \ref{Main theorem} we do not require item $2c$ from Lemma \ref{Important lemma}. This is because our IFS is equicontractive and so the probability vector $\p$ is the uniform vector $(\A^{-1})_{a\in \A}$. Therefore we know exactly how the products in item $2c$ of Lemma \ref{Important lemma} will behave. It is instructive to think that the proof of Statement $3$ follows the proof of Statement $2$ without the introduction of the parameter $\epsilon$.}.

\subsection{Proof of Statement 2 from Theorem \ref{Main theorem}}
Before moving on to our proof of Statement $2,$ it is useful to record for later reference a number of technical results. We start by recalling a well known lemma.

\begin{lemma}
	\label{Erdos lemma}
	Let $(X,A,\m)$ be a finite measure space and $E_n\in A$ be a sequence of 
	sets such that $\sum_{n=1}^{\infty}\m(E_n)=\infty.$ Then
	$$\m\left(\limsup_{n\to\infty} E_{n}\right)\geq \limsup_{Q\to\infty}\frac{(\sum_{n=1}^{Q}\m(E_{n}))^{2}}{\sum_{n,m=1}^{Q}\m(E_{n}\cap E_m)}.$$
\end{lemma}Lemma \ref{Erdos lemma} is due to Kochen and Stone \cite{KocSto}. For a proof of this lemma see either \cite[Lemma 2.3]{Har} or \cite[Lemma 5]{Spr}. The following density lemma has been phrased for our purposes and follows from more general results of Rigot \cite{Rig}.
\begin{lemma}
	\label{Density lemma}
	Let $\m$ be a Bernoulli measure on $\A^{\mathbb{N}}$ and $E\subset \A^{\mathbb{N}}$. Suppose that there exists $c>0$ such that for each finite word $\c\in \cup_{n=1}^{\infty}\A^{n}$ we have 
	$$\m([\c]\cap E)\geq c\cdot \m([\c]).$$ Then $\m(E)=1$.
\end{lemma}

For the rest of this section we fix an IFS $\Phi$ satisfying $$h_{\p}<-2\log \sum_{a\in \A}p_{a}^{2}.$$ We emphasise that throughout this section $\p$ will always be the probability vector corresponding to $(r_{a}^{\dim_{S}(\Phi)}).$ It follows from this inequality that we can pick $\epsilon>0$ sufficiently small such that 
\begin{equation}
\label{technical1}
\frac{-2\chi_{\Phi,\p}}{h_{\p}}<\frac{-\chi_{\Phi,\p}-\epsilon}{-\log \sum_{a\in \A}p_{a}^2}
\end{equation}and
\begin{equation}
\label{technical3}
h_{\p} +\epsilon<-2\log \sum_{a\in A}p_{a}^{2}.
\end{equation}
For the rest of this section we fix $\epsilon$ to be sufficiently small such that \eqref{technical1} and \eqref{technical3} are both satisfied. Let $\Psi:\cup_{n=1}^{\infty}\A^n\to [0,\infty)$ be an arbitrary function of the form $\Psi(\a)=Diam(X_{\a}) g(|\a|)$ for some non-increasing $g:\mathbb{N}\to [0,\infty).$ We also assume that $g$ is such that  $$\sum_{n=1}^{\infty}\sum_{\a\in \A^n} n\cdot (Diam(X_\a)g(n))^{\dim_{S}(\Phi)}=\infty.$$ As previously remarked, this divergence condition is equivalent to 
\begin{equation}
\label{DivergenceA}
\sum_{n=1}^{\infty}n\cdot g(n)^{\dim_{S}(\Phi)}=\infty.
\end{equation} The following two lemmas allow us to replace $g$ with a function whose decaying behaviour we know more about.
\begin{lemma}
	\label{g1lemma}
	Assume that $g$ is a non-increasing function satisfying \eqref{DivergenceA}. Let $g_{1}:\mathbb{N}\to [0,\infty)$ be given by $$g_{1}(n)=\min \left\{g(n),\frac{1}{n^{2/\dim_{S}(\Phi)}}\right\}.$$ Then $$\sum_{n=1}^{\infty}n\cdot g_{1}(n)^{\dim_{S}(\Phi)}=\infty.$$
\end{lemma}
\begin{proof}
Our argument is an adaptation of the proof of Lemma 6.3 from \cite{Bug}. Since $g$ is non-increasing, and so is the sequence $(n^{-2/\dim_{S}(\Phi)})$, we see that the function $g_{1}$ is also non-increasing. Let us now suppose that
\begin{equation}
\label{Dummy convergence}
\sum_{n=1}^{\infty} n\cdot g_{1}(n)^{\dim_{S}(\Phi)}<\infty.
\end{equation} Since $g_{1}$ is non-increasing, for any $m$ sufficiently large we have $$m^2g_{1}(m)^{\dim_{S}(\Phi)}\leq 10\sum_{n=\lfloor m/2\rfloor }^{m}n\cdot g_{1}(n)^{\dim_{S}(\Phi)}.$$ Equation \eqref{Dummy convergence} then implies that $$\lim_{m\to\infty}\sum_{n=\lfloor m/2\rfloor }^{m}n\cdot g_{1}(n)^{\dim_{S}(\Phi)}=0.$$ Combining the two equations above, we may conclude that $$g_{1}(m)<\frac{1}{m^{2/\dim_{S}(\Phi)}}$$ for all $m$ sufficiently large. This means that $g_{1}(n)=g(n)$ for $n$ sufficiently large. Therefore by \eqref{Dummy convergence} we must have $$\sum_{n=1}^{\infty} n\cdot g(n)^{\dim_{S}(\Phi)}<\infty.$$ This contradicts our initial assumption that $g$ satisfies \eqref{DivergenceA}. Therefore we must have\footnote{This is the only part in our proof of Statement $2$ where the assumption $g$ is non-increasing is used. The proof of Statement $3$ differs here in that we define $g_{1}:\mathbb{N}\to [0,\infty)$ by $g_{1}(n)=\min\{g(n),\frac{1}{n^{1/\dim_{S}(\Phi)}}\}.$ Then the appropriate analogue of Lemma \ref{g1lemma} holds for any $g$ satisfying \eqref{DivergenceA}. This is why Statement $3$ holds for arbitrary $g,$ not just those $g$ that are non-increasing.}. 
\begin{equation*}
\sum_{n=1}^{\infty} n\cdot g_{1}(n)^{\dim_{S}(\Phi)}=\infty.
\end{equation*}
\end{proof} 

\begin{lemma}
	\label{g2lemma}
	Assume that $g$ is a non-increasing function satisfying \eqref{DivergenceA} and let $g_1$ be as in Lemma \ref{g1lemma}. Let $g_{2}:\mathbb{N}\to[0,\infty)$ be given by 
	\[g_{2}(n) = \left\{ \begin{array}{ll}
	g_{1}(n) & \mbox{if $g_{1}(n)\geq \frac{1}{n^{4/\dim_{S}(\Phi)}}$};\\
	0 & \mbox{if $g_{1}(n)< \frac{1}{n^{4/\dim_{S}(\Phi)}}$}.\end{array} \right. \] Then $$\sum_{n=1}^{\infty} n\cdot g_{2}(n)^{\dim_{S}(\Phi)}=\infty.$$
\end{lemma}
\begin{proof}
	This follows from Lemma \ref{g1lemma} and the fact $\sum_{n=1}^{\infty}n\cdot \left(\frac{1}{n^{4/\dim_{S}(\Phi)}}\right)^{\dim_{S}(\Phi)}<\infty.$
\end{proof}


Let $g_{2}$ be as in Lemma \ref{g2lemma}. We now define $\Psi_{2}:\cup_{n=1}^{\infty}\A^n\to [0,\infty)$ by $\Psi_{2}(\a)=Diam(X_{\a})g_{2}(|\a|).$ Since $g_{2}(n)\leq g(n)$ for all $n$ it follows that $W_{\Phi}(\Psi_{2})\subset W_{\Phi}(\Psi)$. Therefore to prove Statement $2$ of Theorem \ref{Main theorem} it is sufficient to show that $\mu(W_{\Phi}(\Psi_{2}))=1.$

Let $\c\in \cup_{n=1}^{\infty}\A^n$ be arbitrary and fixed. We will show that 
\begin{equation}
\label{WTS}
\m([\c]\cap \pi^{-1}(W_{\Phi}(\Psi_{2}))\geq c\cdot \m([\c])
\end{equation} for some $c$ independent of $\c$. Lemma \ref{Density lemma} then implies that $\m(\pi^{-1}(W_{\Phi}(\Psi_{2}))=1$. Since $\mu=\pi\m$ this implies that $\mu(W_{\Phi}(\Psi_{2}))=1.$ Therefore to complete our proof it suffices to show that \eqref{WTS} holds.

Let us fix an $N$ sufficiently large so that Lemma \ref{Important lemma} applies for all $n\geq N$ for our choice of $\epsilon.$ We may also assume that $N$ is sufficiently large so that
\begin{equation}
\label{Anna1}
\frac{1}{n^{2/\dim_{S}(\Phi)}}<e^{k_n(-\chi_{\Phi,\p}-\epsilon)}
\end{equation}
for $n\geq N$, and so that there exists $\gamma\in(1,2)$ for which 
\begin{equation}
\label{Anna2}
e^{(h_{\p}+\epsilon)k_{n}}\leq n^{\gamma}
\end{equation}for $n\geq N$. The existence of $\gamma,$ and the fact that \eqref{Anna1} and \eqref{Anna2} are satisfied for $n$ sufficiently large, follows from \eqref{technical1}, \eqref{technical3} and the fact $\dim_{S}(\Phi)=\frac{h_{\p}}{\chi_{\Phi,\p}}.$.

 Let $n\geq N$, for each $\a\in Good(n,\epsilon)$ and $l\in W_{\a}$ we consider the ball $$B\left(\pi(\c a_{1}\ldots a_{l}(a_{l+1}\ldots a_{n})^{\infty}),Diam(X_{\c \a})g_{2}(|\c|+n)\right).$$ If $g_{2}(|\c |+n)\neq 0$ then there exists $h_{\a,l}\in\mathbb{N}$ and $l+1\leq j_{\a,l}\leq n$ such that 
\begin{equation}
\label{inclusion}
X_{\c a_{1}\ldots a_{l}(a_{l+1}\ldots a_{n})^{h_{\a,l}}a_{l+1}\ldots a_{j_{\a,l}}}\subseteq B\left(\pi(\c a_{1}\ldots a_{l}(a_{l+1}\ldots a_{n})^{\infty}),Diam(X_{\c \a})g_{2}(|\c |+n)\right)
\end{equation} and 
\begin{equation}
\label{Comparable diameter}
\min_{a\in \A}r_{a}\cdot Diam(X_{\c \a})g_{2}(|\c|+n)\leq Diam(X_{\c a_{1}\ldots a_{l}(a_{l+1}\ldots a_{n})^{h_{\a,l}}a_{l+1}\ldots a_{j_{\a,l}}})<    Diam(X_{\c \a})g_{2}(|\c |+n).
\end{equation} It follows from the fact $\p=(r_{\a}^{\dim_{S}(\Phi)})$ and $\m$ is the Bernoulli measure on $\A^{\N}$ corresponding to $\p$, that for any word $\a\in \cup_{n=1}^{\infty}\A^n$ we have 
\begin{equation}
\label{measure diameter}
\m([\a])\asymp Diam(X_{\a})^{\dim_{S}(\Phi)}.
\end{equation}Combining \eqref{Comparable diameter} together with \eqref{measure diameter} we can deduce that
\begin{equation}
\label{comparable measure}
\m([\c a_{1}\ldots a_{l}(a_{l+1}\ldots a_{n})^{h_{\a,l}}a_{l+1}\ldots a_{j_{\a,l}}])\asymp \m([\c \a])g_{2}(|\c |+n)^{\dim_{S}(\Phi)}.
\end{equation}
We will use the cylinder sets $[\c a_{1}\ldots a_{l}(a_{l+1}\ldots a_{n})^{h_{\a,l}}a_{l+1}\ldots a_{j_{\a,l}}]$ to prove that \eqref{WTS} holds. 
Before doing that it is useful to prove some properties of the parameters $h_{\a,l}$ and $j_{\a,l}$. 
 \begin{lemma}
 	\label{boring1}
 Let $n\geq N$ be such that $g_{2}(|\c |+n)\neq 0,$ and let $\a\in Good(n,\epsilon)$ and $l\in W_{\a}.$ For $h_{\a,l}$ and $j_{\a,l}$ as defined above, if $h_{\a,l}=1$ then $j_{\a,l}>l+k_{n}.$
 \end{lemma}
\begin{proof}
If $h_{\a,l}=1$ then $$Diam(X_{\c a_{1}\ldots a_{l}(a_{l+1}\ldots a_{n})^{h_{\a,l}}a_{l+1}\ldots a_{j_{\a,l}}})\leq   Diam(X_{\c \a})g_{2}(|\c |+n)$$ implies $$\prod_{i=1}^{j_{\a,l}-l}r_{a_{l+i}}\leq g_{2}(|\c |+n).$$ By Lemma \ref{g1lemma} and Lemma \ref{g2lemma} we know that $g_{2}(|\c|+n)\leq n^{-2/\dim_{S}(\Phi)}$. Therefore
\begin{equation}
\label{Mario}
\prod_{i=1}^{j_{\a,l}-l}r_{a_{l+i}}\leq  \frac{1}{n^{2/\dim_{S}(\Phi)}}.
\end{equation} Importantly, by $2c$ from Lemma \ref{Important lemma} we know that 
\begin{equation}
\label{Luigi}
\prod_{i=1}^{k_{n}}r_{a_{l+i}}\geq e^{k_n(-\chi_{\Phi,\p}-\epsilon)}.
\end{equation}
 Equation \eqref{Anna1} states that  $$\frac{1}{n^{2/\dim_{S}(\Phi)}}<e^{k_n(-\chi_{\Phi,\p}-\epsilon)}$$ for $n\geq N.$ It follows therefore from \eqref{Mario} and \eqref{Luigi} that $$\prod_{i=1}^{j_{\a,l}-l}r_{a_{l+i}}< \prod_{i=1}^{k_{n}}r_{a_{l+i}}.$$ Therefore we must have $j_{\a,l}> l+k_{n}$. 
 \end{proof}
 If we combine 2b. from Lemma \ref{Important lemma} together with Lemma \ref{boring1}, we may conclude the following lemma.

\begin{lemma}
\label{Disjoint lemma}
Assume that $n\geq N$ is such that $g_{2}(|\c|+n)\neq 0$ and let $\a\in Good(n,\epsilon)$. If $l,l'\in W_{\a}$ and $l\neq l'$ then $$\c a_{1}\ldots a_{l}(a_{l+1}\ldots a_{n})^{h_{\a,l}}a_{l+1}\ldots a_{j_{\a,l}}\neq \c a_{1}\ldots a_{l'}(a_{l'+1}\ldots a_{n})^{h_{\a,l'}}a_{l'+1}\ldots a_{j_{\a,l'}}.$$
\end{lemma}

\begin{lemma}
	\label{boring2}
Let $n\geq N$ be such that $g_{2}(|\c |+n)\neq 0,$ and let $\a\in Good(n,\epsilon)$ and $l\in W_{\a}.$ There exists $C=C(\c)$ such that for $h_{\a,l}$ and $j_{\a,l}$ as defined above, we have
$$(n-l)(h_{\a,l}-1)+ j_{\a,l}-l< C\log n.$$
\end{lemma}
\begin{proof}
If $g_{2}(|\c|+n)\neq 0$ then by Lemma \ref{g2lemma} we know that it must satisfy $g_{2}(|\c|+n)\geq \frac{1}{(|\c|+n)^{4/\dim_{S}(\Phi)}}.$ Equation \eqref{Comparable diameter} then implies that if $g_{2}(|\c|+n)\neq 0$ then $$\frac{\min_{\a\in A}r_{a}}{(|\c|+n)^{4/\dim_{S}(\Phi)}}\leq \left(\prod_{i=1}^{n-l}r_{a_{l+i}}\right)^{h_{\a,l}-1} \cdot \prod_{i=1}^{j_{\a,l}-l}r_{a_{l+i}}.$$ This in turn implies that 
$$\frac{\min_{\a\in A}r_{a}}{(|\c|+n)^{4/\dim_{S}(\Phi)}}\leq \left(\max_{a\in A}r_{a}\right)^{(n-l)(h_{\a,l}-1)+ j_{\a,l}-l}.$$ Taking logarithms and then manipulating the resulting expression, one can show that the above implies that there exists $C=C(\c)$ such that 
\begin{equation*}
(n-l)(h_{\a,l}-1)+ j_{\a,l}-l< C\log n.
\end{equation*}
\end{proof}

For each $n\geq N$ such that $g_{2}(|\c|+n)\neq 0$ we let $$E_{n}:=\bigcup_{a\in Good(n,\epsilon)}\bigcup_{l\in W_{\a}}[\c a_{1}\ldots a_{l}(a_{l+1}\ldots a_{n})^{h_{\a,l}}a_{l+1}\ldots a_{j_{\a,l}}].$$ Lemma \ref{Disjoint lemma} tells us that any pair of cylinder sets in this union are disjoint. If $n\geq N$ is such that $g_{2}(|\c|+n)= 0$ then set $E_{n}=\emptyset$. Importantly \eqref{inclusion} implies that $$\limsup_{n\to\infty} E_{n}\subset [\c]\cap \pi^{-1}(W_{\Phi}(\Psi_{2})).$$  Therefore to prove \eqref{WTS} it is sufficient to show that 
\begin{equation}
\label{WTS2}
\m\left(\limsup_{n\to\infty} E_{n}\right)\gg \m([\c]).
\end{equation}We will prove that \eqref{WTS2} holds using Lemma \ref{Erdos lemma}. Before that it is necessary to check that the hypothesis of this lemma are satisfied.

\begin{lemma}
	\label{Divergence lemma}
For $n\geq N$ we have	$\m(E_n)\asymp m([\c])\cdot n\cdot g_{2}(|\c|+n)^{\dim_{S}(\Phi)}.$
\end{lemma}
\begin{proof}
This lemma is obviously true if $n$ is such that $g_{2}(|\c|+n)=0$. As such we restrict our attention to those $n\geq N$ for which $g_{2}(|\c|+n)\neq 0$. Recall that by Lemma \ref{Disjoint lemma}, for distinct $l,l'\in W_{\a}$ we have $\c a_{1}\ldots a_{l}(a_{l+1}\ldots a_{n})^{h_{\a,l}}a_{l+1}\ldots a_{j_{\a,l}}\neq \c a_{1}\ldots a_{l'}(a_{l'+1}\ldots a_{n})^{h_{\a,l'}}a_{l'+1}\ldots a_{j_{\a,l'}}$. Therefore we have
\begin{align*}
\m(E_n)&=\m\left(\bigcup_{\a\in Good(n,\epsilon)}\bigcup_{l\in W_{\a}}[\c a_{1}\ldots a_{l}(a_{l+1}\ldots a_{n})^{h_{\a,l}}a_{l+1}\ldots a_{j_{\a,l}}]\right)\\
&=\sum_{\a\in Good(n,\epsilon)}\sum_{l\in W_{\a}}\m([\c a_{1}\ldots a_{l}(a_{l+1}\ldots a_{n})^{h_{\a,l}}a_{l}\ldots a_{j_{\a,l}}])\\
&\stackrel{\eqref{measure diameter}}{\asymp}\sum_{\a\in Good(n,\epsilon)}\sum_{l\in W_{\a}}Diam(X_{\c a_{1}\ldots a_{l}(a_{l+1}\ldots a_{n})^{h_{\a,l}}a_{l+1}\ldots a_{j_{\a,l}}})
^{\dim_{S}(\Phi)}\\
&\stackrel{\eqref{Comparable diameter}}{\asymp} \sum_{\a\in Good(n,\epsilon)}\sum_{l\in W_{\a}}(Diam(X_{c\a})g_{2}(|\c |+n))^{\dim_{S}(\Phi)}\\
&\stackrel{Lemma\, \ref{Important lemma}}{\asymp}  n\cdot g_{2}(|\c|+n)^{\dim_{S}(\Phi)}\sum_{\a\in Good(n,\epsilon)}Diam(X_{\c\a})^{\dim_{S}(\Phi)}\\
&\stackrel{\eqref{measure diameter}}{\asymp}  n\cdot g_{2}(|\c|+n)^{\dim_{S}(\Phi)}\sum_{\a\in Good(n,\epsilon)}\m([\c\a])\\
&= \m([\c])\cdot n\cdot g_{2}(|\c|+n)^{\dim_{S}(\Phi)}\sum_{\a\in Good(n,\epsilon)}\m([\a])\\
&\stackrel{Lemma\, \ref{Important lemma}}{\asymp} \m([\c])\cdot n\cdot g_{2}(|\c|+n)^{\dim_{S}(\Phi)}.
\end{align*}
\end{proof}


It follows from Lemma \ref{g2lemma} and Lemma \ref{Divergence lemma} that $\sum_{n=1}^{\infty}\m(E_n)=\infty$. So our sequence of sets $(E_n)$ satisfies the hypothesis of Lemma \ref{Erdos lemma}. To complete our proof we need to get good upper bounds for $\m(E_n\cap E_m)$. We restrict our attention to those $n$ and $m$ satisfying $n<m,$ $g_2(|\c|+n)\neq 0,$ and $g_{2}(|\c|+m)\neq 0$. For these $n$ and $m$ we see that
\begin{align*}
\m(E_n\cap E_{m})=\sum_{\a\in Good(n,\epsilon)}\sum_{l\in W_{\a}}\m([\c a_{1}\ldots a_{l}(a_{l+1}\ldots a_{n})^{h_{\a,l}}a_{l+1}\ldots a_{j_{\a,l}}]\cap E_{m}).
\end{align*} The following proposition gives good upper bounds for the terms in this summand. The parameter $C$ in the statement of this proposition is the same $C$ as in Lemma \ref{boring2}.

\begin{prop}
	\label{Measureprop}
Let $n,m\geq N$ be such that $n<m$, $g_2(|\c|+n)\neq 0,$ and $g_{2}(|\c|+m)\neq 0.$ Then for $\a\in Good(n,\epsilon)$ and $l\in W_{\a}$ the following holds:
\begin{enumerate}
	\item If $n<m\leq n+C\log n$ then 
	\begin{align*}
	&\m([\c a_{1}\ldots a_{l}(a_{l+1}\ldots a_{n})^{h_{\a,l}}a_{l+1}\ldots a_{j_{\a,l}}]\cap E_{m})\\
	\ll &\m([\c \a])e^{(h_{\p}+\epsilon) k_{m}}g_{2}(|\c|+n)^{\dim_{S}(\Phi)} g_{2}(|\c|+m)^{\dim_{S}(\Phi)}	+ \m([\c \a])(\max_{a\in \A}p_{a})^{m-n}g_{2}(|\c|+m)^{\dim_{S}(\Phi)}\\
	+&m\cdot \m([\c \a]) g_{2}(|\c|+m)^{\dim_{S}(\Phi)}g_{2}(|\c|+n)^{\dim_{S}(\Phi)}.
	\end{align*}
	\item If $m> C\log n$ then $$\m([\c a_{1}\ldots a_{l}(a_{l+1}\ldots a_{n})^{h_{\a,l}}a_{l+1}\ldots a_{j_{\a,l}}]\cap E_{m})\ll m\cdot \m([\c \a]) g_{2}(|\c|+m)^{\dim_{S}(\Phi)}g_{2}(|\c|+n)^{\dim_{S}(\Phi)}.$$
\end{enumerate}
\end{prop}
	
\begin{proof}
We prove each statement separately.\\

\noindent \textbf{Proof of Statement 1.} Assume that $n<m\leq n+C\log n$.
Let $\a\in Good(n,\epsilon)$ and $l\in W_{\a}$. If $m\leq l +h_{\a,l}(n-l)+(j_{\a,l}-l)$ then at most one $\b\in Good(m,\epsilon)$ is such that $[\c \b]$ non-empty intersection with $[\c a_{1}\ldots a_{l}(a_{l+1}\ldots a_{n})^{h_{\a,l}}a_{l+1}\ldots a_{j_{\a,l}}].$ Let us assume that such a $\b$ exists. Otherwise $\m([\c a_{1}\ldots a_{l}(a_{l+1}\ldots a_{n})^{h_{\a,l}}a_{l+1}\ldots a_{j_{\a,l}}]\cap E_{m})=0$ and our upper bound holds trivially. In this case we see that
\begin{align*}
&\m([\c a_{1}\ldots a_{l}(a_{l+1}\ldots a_{n})^{h_{\a,l}}a_{l+1}\ldots a_{j_{\a,l}}]\cap E_{m})\\
=& \sum_{l'\in W_{\b}}\m([\c a_{1}\ldots a_{l}(a_{l+1}\ldots a_{n})^{h_{\a,l}}a_{l+1}\ldots a_{j_{\a,l}}]\cap [\c b_{1}\ldots b_{l'}(b_{l'+1}\ldots b_{m})^{h_{\b,l'}}b_{l'+1}\ldots b_{j_{\b,l'}}]).
\end{align*} 
Lemma \ref{boring1} implies that if $$\m([\c a_{1}\ldots a_{l}(a_{l+1}\ldots a_{n})^{h_{\a,l}}a_{l+1}\ldots a_{j_{\a,l}}]\cap [\c b_{1}\ldots b_{l'}(b_{l'+1}\ldots b_{m})^{h_{\b,l'}}b_{l'+1}\ldots b_{j_{\b,l'}}])\neq 0$$ then we must have $$[\c a_{1}\ldots a_{l}(a_{l+1}\ldots a_{n})^{h_{\a,l}}a_{l+1}\ldots a_{j_{\a,l}}]\cap [\c b_{1}\ldots b_{m}b_{l'+1}\ldots b_{l'+k_m}]\neq \emptyset.$$
By Lemma \ref{Important lemma} we know that for each $l'\in W_{\b}$ the cylinder set $[\c b_{1}\ldots b_{m}b_{l'+1}\ldots b_{l'+k_m}]$ satisfies $$\m([\c b_{1}\ldots b_{m}b_{l'+1}\ldots b_{l'+k_m}])\geq \m([\c\b])\cdot e^{k_{m}(-h_{\p}-\epsilon)}.$$ Therefore, by \eqref{comparable measure} and a measure argument we have
\begin{align}
\label{intersection count}
\#\{l'\in W_{b} :[&\c a_{1}\ldots a_{l}(a_{l+1}\ldots a_{n})^{h_{\a,l}}a_{l+1}\ldots a_{j_{\a,l}}]\cap [\c b_{1}\ldots b_{m}b_{l'+1}\ldots b_{l'+k_m}]\neq \emptyset\}\nonumber\\
&\ll \frac{\m([\c \a])g_{2}(|\c|+n)^{\dim_{S}(\Phi)}}{\m([\c \b]) e^{k_{m}(-h_{\p}-\epsilon)}}+1.
\end{align}
Applying the above observations we see that
\begin{align*}
&\sum_{l'\in W_{\b}}\m([\c a_{1}\ldots a_{l}(a_{l+1}\ldots a_{n})^{h_{\a,l}}a_{l+1}\ldots a_{j_{\a,l}}]\cap [\c b_{1}\ldots b_{l'}(b_{l'+1}\ldots b_{m})^{h_{\b,l'}}b_{l'+1}\ldots b_{j_{\b,l'}}])\\
\leq &\sum_{\stackrel{l'\in W_{\b}}{[\c a_{1}\ldots a_{l}(a_{l+1}\ldots a_{n})^{h_{\a,l}}a_{l+1}\ldots a_{j_{\a,l}}]\cap [\c b_{1}\ldots b_{m}b_{l'+1}\ldots b_{l'+k_m}]\neq \emptyset}}\m([\c b_{1}\ldots b_{l'}(b_{l'+1}\ldots b_{m})^{h_{\b,l'}}b_{l'+1}\ldots b_{j_{\b,l'}}])\\
\stackrel{\eqref{comparable measure}}\ll&\sum_{\stackrel{l'\in W_{\b}}{[\c a_{1}\ldots a_{l}(a_{l+1}\ldots a_{n})^{h_{\a,l}}a_{l+1}\ldots a_{j_{\a,l}}]\cap [\c b_{1}\ldots b_{m}b_{l'+1}\ldots b_{l'+k_m}]\neq \emptyset}} \m([\c\b])g_{2}(|\c|+m)^{\dim_{S}(\Phi)}\\
\stackrel{\eqref{intersection count}}{\ll} & \m([\c \a])e^{(h_{\p}+\epsilon)k_{m}}g_{2}(|\c|+n)^{\dim_{S}(\Phi)} g_{2}(|\c|+m)^{\dim_{S}(\Phi)}+\m([\c\b])g_{2}(|\c|+m)^{\dim_{S}(\Phi)}.
\end{align*}
Because $\b$ must have $\a$ as a prefix we see that 
\begin{equation}
\label{expdecay}
\m([\c\b])\leq \m([\c \a])(\max_{a\in \A}p_{a})^{m-n}.
\end{equation}
Substituting \eqref{expdecay} into the last line in the above, we have shown that if $m\leq l +h_{\a,l}(n-l)+(j_{\a,l}-l)$ then\footnote{In the proof of Statement $3$ from Theorem \ref{Main theorem} we know that $\p=(\A^{-1})_{a\in \A},$ and as such we can make more precise statements about the measure of cylinders. Indeed in the above we do not need to introduce the parameter $\epsilon$ and \eqref{intersection count} holds with $\epsilon=0$. This means that we can strengthen \eqref{Case1Abound} to $\m([\c a_{1}\ldots a_{l}(a_{l+1}\ldots a_{n})^{h_{\a,l}}a_{l+1}\ldots a_{j_{\a,l}}]\cap E_{m})\ll \m([\c \a])e^{h_{\p}k_{m}}g_{2}(|\c|+n)^{\dim_{S}(\Phi)} g_{2}(|\c|+m)^{\dim_{S}(\Phi)}+\m([\c \a])(\max_{a\in \A}p_{a})^{m-n}g_{2}(|\c|+m)^{\dim_{S}(\Phi)}.$ Which by the definition of $k_{m}$ implies $\m([\c a_{1}\ldots a_{l}(a_{l+1}\ldots a_{n})^{h_{\a,l}}a_{l+1}\ldots a_{j_{\a,l}}]\cap E_{m})\ll \m([\c \a])m g_{2}(|\c|+n)^{\dim_{S}(\Phi)} g_{2}(|\c|+m)^{\dim_{S}(\Phi)}+\m([\c \a])(\max_{a\in \A}p_{a})^{m-n}g_{2}(|\c|+m)^{\dim_{S}(\Phi)}.$ The rest of the proof follows identically.}
\begin{align}
\label{Case1Abound}
&m([\c a_{1}\ldots a_{l}(a_{l+1}\ldots a_{n})^{h_{\a,l}}a_{l+1}\ldots a_{j_{\a,l}}]\cap E_{m})\nonumber\\
\ll &\m([\c \a])e^{(h_{\p}+\epsilon)k_{m}}g_{2}(|\c|+n)^{\dim_{S}(\Phi)} g_{2}(|\c|+m)^{\dim_{S}(\Phi)}+ \m([\c \a])(\max_{a\in \A}p_{a})^{m-n}g_{2}(|\c|+m)^{\dim_{S}(\Phi)}.
\end{align}
Now suppose that $m>  l +h_{\a,l}(n-l)+(j_{\a,l}-l).$ In this case 
\begin{align}
\label{Hampster}
&\m([\c a_{1}\ldots a_{l}(a_{l+1}\ldots a_{n})^{h_{\a,l}}a_{l+1}\ldots a_{j_{\a,l}}]\cap E_{m})\nonumber\\
=& \sum_{\stackrel{\b\in Good(m,\epsilon)}{\b \textrm{ begins with }a_{1}\ldots a_{l}(a_{l+1}\ldots a_{n})^{h_{\a,l}}a_{l+1}\ldots a_{j_{\a,l}}}}\sum_{l'\in W_{\b}}\m([\c b_{1}\ldots b_{l'}(b_{l'+1}\ldots b_{m})^{h_{\b,l'}}b_{l'+1}\ldots b_{j_{\b,l'}}])\nonumber\\
\stackrel{\eqref{comparable measure}}{\ll}&\sum_{\stackrel{\b\in Good(m,\epsilon)}{\b \textrm{ begins with }a_{1}\ldots a_{l}(a_{l+1}\ldots a_{n})^{h_{\a,l}}a_{l+1}\ldots a_{j_{\a,l}}}}\sum_{l'\in W_{\b}}\m([\c \b])g_{2}(|\c|+m)^{\dim_{S}(\Phi)}\nonumber\\
\leq &m\cdot g_{2}(|\c|+m)^{\dim_{S}(\Phi)}\sum_{\stackrel{\b\in Good(m,\epsilon)}{\b \textrm{ begins with }a_{1}\ldots a_{l}(a_{l+1}\ldots a_{n})^{h_{\a,l}}a_{l+1}\ldots a_{j_{\a,l}}}}\m([\c \b])\nonumber\\
\leq & m\cdot g_{2}(|\c|+m)^{\dim_{S}(\Phi)}\sum_{\stackrel{\b\in \A^m}{\b \textrm{ begins with }a_{1}\ldots a_{l}(a_{l+1}\ldots a_{n})^{h_{\a,l}}a_{l+1}\ldots a_{j_{\a,l}}}}\m([\c \b])\nonumber\\
=& m\cdot g_{2}(|\c|+m)^{\dim_{S}(\Phi)}\cdot \m([\c a_{1}\ldots a_{l}(a_{l+1}\ldots a_{n})^{h_{\a,l}}a_{l+1}\ldots a_{j_{\a,l}}])\nonumber \\
\stackrel{\eqref{comparable measure}}{\ll}& m\cdot \m([\c \a]) g_{2}(|\c|+m)^{\dim_{S}(\Phi)}g_{2}(|\c|+n)^{\dim_{S}(\Phi)}.
\end{align}\\
Adding together the upper bounds obtained in \eqref{Case1Abound} and \eqref{Hampster} we obtain the desired upper bound which holds for all $m$ satisfying $n<m\leq n+C\log n$.\\

\noindent\textbf{Proof of Statement 2.} Assume $m>n+C\log n$. Let $\a\in Good(n,\epsilon)$ and $l\in W_{\a}$. If $m>n+C\log n$ then by Lemma \ref{boring2} we must have $m>l +h_{\a,l}(n-l)+(j_{\a,l}-l)$. In which case the same argument as is used in the proof of the second part of Statement $1$ applies and we have the desired bound
\begin{equation*}
\label{Case2bound}
\m([\c a_{1}\ldots a_{l}(a_{l+1}\ldots a_{n})^{h_{\a,l}}a_{l+1}\ldots a_{j_{\a,l}}]\cap E_{m})\leq m\cdot \m([\c \a])g_{2}(|\c|+m)^{\dim_{S}(\Phi)}g_{2}(|\c|+n)^{\dim_{S}(\Phi)}.
\end{equation*}\\

\end{proof}

Equipped with Proposition \ref{Measureprop} we will now prove the following statement.
\begin{prop}
\label{workingprop}
There exists a constant $C_{1}=C_{1}(\c)$ such that  $$\sum_{n,m=N}^{Q}\m(E_{n}\cap E_{m})\ll \m([\c])\left(\sum_{n=N}^{Q}n\cdot g_{2}(|\c|+n)^{\dim_{S}(\Phi)}+\left(\sum_{n=N}^{Q}n\cdot g_{2}(|\c|+n)^{\dim_{S}(\Phi)}\right)^{2}\right)+C_{1}.$$
\end{prop}
\begin{proof}
We start our proof by rewriting $\sum_{n,m=N}^{Q}\m(E_{n}\cap E_m):$
\begin{align*}
\sum_{n,m=N}^{Q}\m(E_{n}\cap E_{m})&=\sum_{n=N}^{Q}\m(E_{n})+2\sum_{n=N}^{Q-1}\sum_{m=n+1}^{Q}\m(E_{n}\cap E_{m})\\
&=\underbrace{\sum_{n=N}^{Q}\m(E_{n})}_{A}+\underbrace{2\sum_{n=N}^{Q-1}\sum_{m=n+1}^{\min\{Q,n+C\log n\}}\m(E_{n}\cap E_{m})}_{B}+\underbrace{2\sum_{n=N}^{Q-1}\sum_{n+C\log n<m\leq Q}\m(E_n\cap E_m)}_{C}.
\end{align*}
We will focus on the three terms A, B, and C individually. By Lemma \ref{Divergence lemma} we have the following bound for term A:
\begin{equation}
\label{part1}
\sum_{n=N}^{Q}\m(E_{n})\asymp\m([\c])\sum_{n=N}^{Q}n\cdot g_{2}(|\c|+n)^{\dim_{S}(\Phi)}.
\end{equation}
Now focusing on the term $B,$ if we apply Statement $1$ from Proposition \ref{Measureprop} we have 
\begin{align*}
&\sum_{n=N}^{Q-1}\sum_{m=n+1}^{\min\{Q,n+C\log n\}}\m(E_{n}\cap E_{m})\\
= & \sum_{\stackrel{n=N}{g_{2}(n)\neq 0}}^{Q-1}\sum_{\stackrel{m=n+1}{g_{2}(m)\neq 0}}^{\min\{Q,n+C\log n\}}\sum_{\a\in Good(n,\epsilon)}\sum_{l\in W_{\a}}\m([\c a_{1}\ldots a_{l}(a_{l+1}\ldots a_{n})^{h_{\a,l}}a_{l+1}\ldots a_{j_{\a,l}}]\cap E_{m})\\
\ll& \underbrace{\sum_{n=N}^{Q-1}\sum_{m=n+1}^{\min\{Q,n+C\log n\}}\sum_{\a\in Good(n,\epsilon)}\sum_{l\in W_{\a} }\m([\c \a])e^{(h_{\p}+\epsilon)k_{m}}g_{2}(|\c|+n)^{\dim_{S}(\Phi)} g_{2}(|\c|+m)^{\dim_{S}(\Phi)}}_{B1}\\
&+\underbrace{\sum_{n=N}^{Q-1}\sum_{m=n+1}^{\min\{Q,n+C\log n\}}\sum_{\a\in Good(n,\epsilon)}\sum_{l\in W_{\a} }\m([\c \a])(\max_{a\in \A}p_{a})^{m-n}g_{2}(|\c|+m)^{\dim_{S}(\Phi)}}_{B2}\\
&+\underbrace{\sum_{n=N}^{Q-1}\sum_{m=n+1}^{\min\{Q,n+C\log n\}}\sum_{\a\in Good(n,\epsilon)}\sum_{l\in W_{\a} }m\cdot\m([\c \a]) g_{2}(|\c|+m)^{\dim_{S}(\Phi)}g_{2}(|\c|+n)^{\dim_{S}(\Phi)}}_{B3}.
\end{align*}
Focusing on the term B1 in the above, we know by \eqref{Anna2} that
$$e^{(h_{\p}+\epsilon)k_{m}}\leq m^{\gamma}$$ for some $\gamma\in(1,2)$. Using this inequality we have
\begin{align*}
&\sum_{n=N}^{Q-1}\sum_{m=n+1}^{\min\{Q,n+C\log n\}}\sum_{\a\in Good(n,\epsilon)}\sum_{l\in W_{\a} }\m([\c \a])e^{(h_{\p}+\epsilon)k_{m}}g_{2}(|\c|+n)^{\dim_{S}(\Phi)} g_{2}(|\c|+m)^{\dim_{S}(\Phi)}\\
\leq &\sum_{n=N}^{Q-1}\sum_{m=n+1}^{\min\{Q,n+C\log n\}}\sum_{\a\in Good(n,\epsilon)}\sum_{l\in W_{\a} }m^{\gamma}\m([\c \a])g_{2}(|\c|+n)^{\dim_{S}(\Phi)} g_{2}(|\c|+m)^{\dim_{S}(\Phi)}\\
\leq & \sum_{n=N}^{Q-1}\sum_{m=n+1}^{\min\{Q,n+C\log n\}}\sum_{\a\in Good(n,\epsilon)}n\cdot m^{\gamma}\m([\c \a])g_{2}(|\c|+n)^{\dim_{S}(\Phi)} g_{2}(|\c|+m)^{\dim_{S}(\Phi)}\\
\leq & \m([\c])\sum_{n=N}^{Q-1}\sum_{m=n+1}^{\min\{Q,n+C\log n\}}n\cdot m^{\gamma} g_{2}(|\c|+n)^{\dim_{S}(\Phi)} g_{2}(|\c|+m)^{\dim_{S}(\Phi)}\\
\leq & \m([\c])\sum_{n=N}^{Q-1}n\cdot g_{2}(|\c|+n)^{\dim_{S}(\Phi)}\sum_{m=n+1}^{\min\{Q,n+C\log n\}}m^{\gamma}\cdot g_{2}(|\c|+m)^{\dim_{S}(\Phi)}
\end{align*} It follows from Lemma \ref{g1lemma} and Lemma \ref{g2lemma} that $g_{2}(|\c|+m)^{\dim_{S}(\Phi)}\leq m^{-2}$ for all $m\in \mathbb{N}$. Therefore we have
\begin{align*}
& \m([\c])\sum_{n=N}^{Q-1}n\cdot g_{2}(|\c|+n)^{\dim_{S}(\Phi)}\sum_{m=n+1}^{\min\{Q,n+C\log n\}}m^{\gamma}\cdot g_{2}(|\c|+m)^{\dim_{S}(\Phi)}\\
\leq & \m([\c])\sum_{n=N}^{Q-1}n\cdot g_{2}(|\c|+n)^{\dim_{S}(\Phi)}\sum_{m=n+1}^{\min\{Q,n+C\log n\}}m^{\gamma-2}\\
\ll & \m([\c])\sum_{n=N}^{Q-1}n\cdot g_{2}(|\c|+n)^{\dim_{S}(\Phi)}\int_{n+1}^{n+C\log n}x^{\gamma-2}\, dx\\
\ll& \m([\c])\sum_{n=N}^{Q-1}n\cdot g_{2}(|\c|+n)^{\dim_{S}(\Phi)}\left((n+C\log n)^{\gamma-1}-(n+1)^{\gamma-1}\right)\\
\stackrel{M.V.T}{\ll}&  \m([\c])\sum_{n=N}^{Q-1}n\cdot g_{2}(|\c|+n)^{\dim_{S}(\Phi)}\cdot \left(C\log n \cdot \frac{1}{n^{2-\gamma}}\right)\\
\ll& \m([\c])\sum_{n=N}^{Q-1}\frac{C\log n}{n^{3-\gamma}}\\
\ll& \m([\c])\sum_{n=1}^{\infty}\frac{C\log n}{n^{3-\gamma}}.\\
\end{align*}
In the penultimate line in the above we used that $g_{2}(|\c|+n)^{\dim_{S}(\Phi)}\leq n^{-2}$. Because $\gamma\in(1,2)$ we know that $\sum_{n=1}^{\infty}\frac{C\log n}{n^{3-\gamma}}<\infty$. Therefore we can assert that there exists a constant $C_{1}=C_{1}(\c)$ so that the term $B1$ satisfies
\begin{equation}
\label{C1bound}
\sum_{n=N}^{Q-1}\sum_{m=n+1}^{\min\{Q,n+C\log n\}}\sum_{\a\in Good(n,\epsilon)}\sum_{l\in W_{\a} }\m([\c \a])e^{(h_{\p}+\epsilon)k_{m}}g_{2}(|\c|+n)^{\dim_{S}(\Phi)} g_{2}(|\c|+m)^{\dim_{S}(\Phi)}\leq C_{1}.
\end{equation}
 Turning our attention to the term $B2$ we see that
\begin{align*}
&\sum_{n=N}^{Q-1}\sum_{m=n+1}^{\min\{Q,n+C\log n\}}\sum_{\a\in Good(n,\epsilon)}\sum_{l\in W_{\a} }\m([\c \a])(\max_{a\in \A}p_{a})^{m-n}g_{2}(|\c|+m)^{\dim_{S}(\Phi)}\\
\ll & \sum_{n=N}^{Q-1}\sum_{m=n+1}^{\min\{Q,n+C\log n\}}\sum_{\a\in Good(n,\epsilon)}n\cdot \m([\c \a])(\max_{a\in \A}p_{a})^{m-n}g_{2}(|\c|+m)^{\dim_{S}(\Phi)}\\
\ll & \m([\c]) \sum_{n=N}^{Q-1}\sum_{m=n+1}^{\min\{Q,n+C\log n\}}n\cdot (\max_{a\in \A}p_{a})^{m-n}g_{2}(|\c|+m)^{\dim_{S}(\Phi)}\\
\ll &\m([\c])\sum_{m=N+1}^{Q}\sum_{n=1}^{m-1}n\cdot (\max_{a\in \A}p_{a})^{m-n}g_{2}(|\c|+m)^{\dim_{S}(\Phi)}\\
= & \m([\c])\sum_{m=N+1}^{Q}g_{2}(|\c|+m)^{\dim_{S}(\Phi)}\sum_{n=1}^{m-1}n\cdot (\max_{a\in \A}p_{a})^{m-n}.
\end{align*}
Now using the fact that $\sum_{n=1}^{m-1}n\cdot (\max_{a\in \A}p_{a})^{m-n}\ll m$ we see that 
\begin{equation*}
\m([\c])\sum_{m=N+1}^{Q}g_{2}(|\c|+m)^{\dim_{S}(\Phi)}\sum_{n=1}^{m-1}n\cdot (\max_{a\in \A}p_{a})^{m-n}\ll \m([\c])\sum_{m=N+1}^{Q}m\cdot g_{2}(|\c|+m)^{\dim_{S}(\Phi)}.
\end{equation*}
So our term $B2$ must satisfy 
\begin{equation}
\label{Hometime}
\sum_{n=N}^{Q-1}\sum_{m=n+1}^{\min\{Q,n+C\log n\}}\sum_{\a\in Good(n,\epsilon)}\sum_{l\in W_{\a} }\m([\c \a])(\max_{a\in \A}p_{a})^{m-n}g_{2}(|\c|+m)^{\dim_{S}(\Phi)}\ll \m([\c])\sum_{n=N}^{Q}n\cdot g_{2}(|\c|+n)^{\dim_{S}(\Phi)}.
\end{equation}
Now focusing on the term $B3$, we have 
\begin{align}
\label{nearlyhometime}
&\sum_{n=N}^{Q-1}\sum_{m=n+1}^{\min\{Q,n+C\log n\}}\sum_{\a\in Good(n,\epsilon)}\sum_{l\in W_{\a} }m\cdot\m([\c \a]) g_{2}(|\c|+m)^{\dim_{S}(\Phi)}g_{2}(|\c|+n)^{\dim_{S}(\Phi)}\nonumber\\
\leq & \sum_{n=N}^{Q-1}\sum_{m=n+1}^{\min\{Q,n+C\log n\}}\sum_{\a\in Good(n,\epsilon)}n\cdot m\cdot \m([\c \a])g_{2}(|\c|+m)^{\dim_{S}(\Phi)}g_{2}(|\c|+n)^{\dim_{S}(\Phi)}\nonumber\\
\leq &\m([\c]) \sum_{n=N}^{Q-1}\sum_{m=n+1}^{\min\{Q,n+C\log n\}}n\cdot m\cdot g_{2}(|\c|+m)^{\dim_{S}(\Phi)}g_{2}(|\c|+n)^{\dim_{S}(\Phi)}\nonumber\\
\leq &\m([\c]) \sum_{n=N}^{Q-1}n\cdot g_{2}(|\c|+n)^{\dim_{S}(\Phi)}\sum_{m=n+1}^{\min\{Q,n+C\log n\}}m\cdot g_{2}(|\c|+m)^{\dim_{S}(\Phi)}\nonumber\\
\leq &\m([\c])\left(\sum_{n=N}^{Q}n\cdot g_{2}(|\c|+n)^{\dim_{S}(\Phi)}\right)^{2}.
\end{align}Combining \eqref{C1bound}, \eqref{Hometime}, and \eqref{nearlyhometime} we have the following bound for the term $B$
\begin{align}
\label{partb}
&2\sum_{n=N}^{Q-1}\sum_{m=n+1}^{\min\{Q,n+C\log n\}}\m(E_{n}\cap E_{m})\nonumber\\
\ll &\m([\c])\left(\sum_{n=N}^{Q}n\cdot g_{2}(|\c|+n)^{\dim_{S}(\Phi)}+\sum_{n=N}^{Q}n\cdot g_{2}(|\c|+n)^{\dim_{S}(\Phi)}\right)^{2}+C_{1}.
\end{align}

By an analogous argument to that used to bound $B3$, by applying Statement $2$ from Proposition \ref{Measureprop} it can also be shown that the term $C$ satisfies 
\begin{equation}
\label{part3}
2\sum_{n=N}^{Q-1}\sum_{n+C\log n<m\leq Q}\m(E_n\cap E_m)\ll \m([\c])\left(\sum_{n=N}^{Q}n\cdot g_{2}(|\c|+n)^{\dim_{S}(\Phi)}\right)^{2}.
\end{equation} Combining \eqref{part1}, \eqref{partb}, and \eqref{part3} we may conclude our desired bound $$\sum_{n,m=N}^{Q}\m(E_{n}\cap E_{m})\ll \m([\c])\left(\sum_{n=N}^{Q}n\cdot g_{2}(|\c|+n)^{\dim_{S}(\Phi)}+\left(\sum_{n=N}^{Q}n\cdot g_{2}(|\c|+n)^{\dim_{S}(\Phi)}\right)^{2}\right)+C_{1}.$$ 
\end{proof}
Combining Proposition \ref{workingprop} together with Lemma \ref{Erdos lemma} and Lemma \ref{Divergence lemma} we may conclude that
\begin{align*}
\m(\limsup_{n\to\infty} E_n)&\geq \limsup_{Q\to\infty}\frac{\left(\sum_{n=N}^{Q}\m(E_{n})\right)^{2}}{\sum_{n,m=N}^{Q}\m(E_{n}\cap E_{m})}\\
&\gg \limsup_{Q\to\infty}\frac{\m([\c])^{2}\left(\sum_{n=N}^{Q}n\cdot g_{2}(|\c|+n)^{\dim_{S}(\Phi)}\right)^2}{\m([\c]\left(\sum_{n=N}^{Q}n\cdot g_{2}(|\c|+n)^{\dim_{S}(\Phi)}+\left(\sum_{n=N}^{Q}n\cdot g_{2}(|\c|+n)^{\dim_{S}(\Phi)}\right)^{2}\right)+C_{1}}\\
&= \limsup_{Q\to\infty}\frac{\m([\c])^{2}\left(\sum_{n=N}^{Q}n\cdot g_{2}(|\c|+n)^{\dim_{S}(\Phi)}\right)^2}{\m([\c]\left(\sum_{n=N}^{Q}n\cdot g_{2}(|\c|+n)^{\dim_{S}(\Phi)}+\left(\sum_{n=N}^{Q}n\cdot g_{2}(|\c|+n)^{\dim_{S}(\Phi)}\right)^{2}\right)}\\
&\gg \m([\c]).
\end{align*}
Thus \eqref{WTS2} holds and our proof of Statement $2$ from Theorem \ref{Main theorem} is complete. We emphasise that in the penultimate line in the above we used the fact that the constant $C_{1}$ does not affect the limit. This is important because the implicit constant in \eqref{WTS2} needs to be independent of $\c$ if we want to apply Lemma \ref{Density lemma}. 

\section{Applications of the mass transference principle}
\label{MTP section}
The mass transference principle of Beresnevich and Velani \cite{BV} is a powerful tool that allows one to derive information on the Hausdorff measure of a limsup set. We do not state it in its full generality, but instead content ourselves with the following which is better suited for our purposes.

Let $X\subset \mathbb{R}^d$. Then $X$ is said to be Ahlfors regular if there exists $C_{1},C_{2}>0$ such that $$C_{1}r^{\dim_{H}(X)}\leq \mathcal{H}^{\dim_{H}(X)}(B(x,r)\cap X)\leq C_{2}r^{\dim_{H}(X)}$$ for all $x\in X$ and $r$ sufficiently small. Given an Ahlfors regular set $X$, a ball $B(x,r)$ in $X$, and $s>0$, we let $B^{s}=B(x,r^{s/\dim_{H}(X)})$. The following theorem is a simplified version of Theorem 3 from \cite{BV}.

\begin{thm}
	\label{Mass transference principle}
	Let $X$ be Ahlfors regular and $(B_l)$ be a sequence of balls in $X$ with radii tending to zero. Let $s>0$ and suppose that for any ball $B$ in $X$ we have $$\mathcal{H}^{\dim_{H}(X)}\left(B\cap \limsup_{l\to \infty}B_{l}^{s}\right)=\mathcal{H}^{\dim_{H}(X)}(B).$$ Then, for any ball $B$ in $X$ $$\mathcal{H}^{s}\left(B\cap \limsup_{l\to\infty}B_l\right)=\mathcal{H}^{s}(B).$$
\end{thm}
It is a well know fact that if an IFS $\Phi$ satisfies the open set condition then the corresponding self-similar set is Ahlfors regular. It is also well known that if $\Phi$ satisfies the open set condition then $\mu$ is equivalent to the restriction of $\mathcal{H}^{\dim_{S}(\Phi)}$ on $X$. Combining these facts together with Theorem \ref{Main theorem} and Theorem \ref{Mass transference principle}, we may deduce the following statement. 

\begin{thm}
	\label{MTP corollary}
Let $\Phi=\{\phi_a\}_{a\in \A}$ be an IFS which satisfies the open set condition. Let $\Psi:\cup_{n=1}^{\infty}\A^{n}\to [0,\infty)$ be given by $\Psi(\a)=Diam(X_{\a})g(|\a|)$ for some function $g:\mathbb{N}\to [0,\infty)$ satisfying $$\sum_{n=1}^{\infty}\sum_{\a\in\A^n}n\cdot (Diam(X_{\a})g(n))^{\dim_{S}(\Phi)}=\infty.$$ Then the following statements are true:
\begin{enumerate}
	\item Assume that $$h_{\p}<-2\log \sum_{a\in \A}p_{a}^{2}$$ and that $g$ is non-increasing. Then for any $t\geq 1$ we have $\mathcal{H}^{\dim_{H}(X)/t}(W_{\Phi}(\Psi^{t}))=\mathcal{H}^{\dim_{H}/t}(X)$.
	\item If $\Phi$ is equicontractive then for any $t\geq 1$ we have $\mathcal{H}^{\dim_{H}(X)/t}(W_{\Phi}(\Psi^{t}))=\mathcal{H}^{\dim_{H}/t}(X)$.
\end{enumerate}
\end{thm}
We conclude this section by mentioning that by following the arguments used in Section \ref{Application section}, one can use Theorem \ref{MTP corollary} to prove a number of statements on the Hausdorff measure of certain limsup sets arising from the study of intrinsic Diophantine Approximation on self-similar sets. We leave the details to the interested reader.\\

\noindent \textbf{Acknowledgements.} The author would like to thank Baowei Wang for his feedback on an initial draft.

\end{document}